\documentclass[a4paper]{amsart}
        \usepackage{latexsym}
        \usepackage{amssymb}
        \usepackage{amsmath}
        \usepackage{amsfonts}
        \usepackage{amsthm}
        \usepackage[hypertexnames=false]{hyperref}
        \ifpdf
          \usepackage[all,knot,poly,2cell,cmtip]{xy}
        \else
          \usepackage[all,knot,poly,2cell,dvips,cmtip]{xy}
        \fi
             \UseAllTwocells
        \usepackage{xspace}

        \usepackage{eucal}

        \numberwithin{equation}{section}
        \theoremstyle{plain}
        \newtheorem{theorem}[equation]{Theorem}
        \newtheorem{corollary}[equation]{Corollary}
        \newtheorem{lemma}[equation]{Lemma}
        
        \newtheorem{proposition}[equation]{Proposition}
        \newtheorem{maintheorem}{Theorem}

        \theoremstyle{definition}
        
        \newtheorem{example}[equation]{Example}
        
        \newtheorem*{example*}{Example}

        \theoremstyle{remark}
        \newtheorem{remark}[equation]{Remark}
        \newtheorem*{remark*}{Remark}

        \newcommand{\suchthat}{\,:\,}
        \newcommand{\itemref}[1]{\eqref{#1}}

        \newcommand{\Z}{\mathbb{Z}}
        \newcommand{\N}{\mathbb{N}}
        \newcommand{\F}{\mathbb{F}} 
        \newcommand{\Q}{\mathbb{Q}}

        \newcommand{\Orb}{\mathcal{O}}   
       \DeclareMathOperator{\spec}{Spec} 
           \newcommand{\et}{\mathrm{\acute{e}t}}

           \newcommand{\lisset}{\mathrm{lis\text{\nobreakdash-}\acute{e}t}}

        \newcommand{\MOD}{\mathsf{Mod}}    

        \DeclareMathOperator{\coker}{coker}
         \DeclareMathOperator{\Hom}{Hom}
        \DeclareMathOperator{\Ext}{Ext}

        \newcommand{\COHO}[1]{\mathcal{H}^{{#1}}}
        \newcommand{\trunc}[1]{\tau^{{#1}}}
        \newcommand{\RDERF}{\mathsf{R}}
        \newcommand{\LDERF}{\mathsf{L}}
        
        \newcommand{\DCAT}{\mathsf{D}}
        \newcommand{\RHom}{\RDERF\!\Hom}
        \newcommand{\SHom}{\mathcal{H}om}
        \newcommand{\SRHom}{\RDERF\SHom}
        \newcommand{\QCOH}{\mathsf{QCoh}}
        \newcommand{\COH}{\mathsf{Coh}}
        \newcommand{\PERF}{\mathsf{Perf}}

        \renewcommand{\bar}[1]{\overline{{#1}}}

        \newcommand{\ID}[1]{\mathrm{Id}_{#1}}

        \newcommand{\tensor}{\otimes}

        \newcommand{\opp}{\circ}
        
                \newcommand{\AB}{\mathsf{Ab}}
                \newcommand{\SETS}{\mathsf{Sets}}

\numberwithin{equation}{section}

\newcommand{\qcsubscript}{\mathrm{qc}} 

\newcommand{\DQCOH}[1][]{\DCAT_{\qcsubscript{#1}}} 

\newcommand{\GL}{\mathrm{GL}} 
\newcommand{\Gm}{\mathbb{G}_m} 
\newcommand{\Ga}{\mathbb{G}_a} 

\newcommand{\hocolim}[1]{\underset{#1}{\mathrm{hocolim}}\,}
\newcommand{\holim}[1]{\underset{#1}{\mathrm{holim}}\,}

\newcommand{\spref}[1]{\href{http://stacks.math.columbia.edu/tag/#1}{#1}}

\makeatletter
\newcommand{\labitem}[2]{%
\def\@itemlabel{(\textbf{#1})}
\item
\def\@currentlabel{\textbf{#1}}\label{#2}}
\makeatother

\usepackage{mathtools}
\theoremstyle{theorem}
\newtheorem*{corollary*}{Corollary}
\theoremstyle{definition}
\newtheorem*{definition*}{Definition}

\title{Further remarks on derived categories of algebraic stacks} \date{May 19, 2022}
\author[J. Hall]{Jack Hall}


\newcommand{\shfcoho}{\mathsf{H}}
\newcommand{\cd}{\mathrm{cd}}
\newcommand{\closure}[1]{\bar{{#1}}}
\renewcommand{\mathbb}{\mathbf}
\DeclareMathOperator{\amp}{amp}
\usepackage{enumitem}
\setlist[enumerate]{font=\normalfont}
\begin{document}
\begin{abstract}
  Let $X$ be an algebraic stack with quasi-affine diagonal of finite type over 
a field $k$ of characteristic $0$. We extend the well-known equivalence
$\DCAT^+(\QCOH(X)) \simeq \DQCOH^+(X)$ to unbounded derived categories. We also prove
that if $X$ is smooth over $k$, then $\DQCOH(X)$ is compactly generated. We
accomplish the former using the descendable algebras of Mathew. We
also establish related results in positive and mixed characteristics.
\end{abstract}
\maketitle
\section{Introduction}\label{S:intro}
Let $X$ be a quasi-compact and quasi-separated scheme. Then 
\begin{enumerate}[label=(\alph*), ref=(\alph*)]
\item \label{EI:gen} $\DQCOH(X)$, the unbounded derived category of
  $\Orb_X$-modules whose cohomology sheaves belong to the abelian
  category of quasi-coherent sheaves $\QCOH(X)$, is compactly
  generated by perfect complexes; and
\item \label{EI:eq} if $X$ has affine diagonal or is noetherian, then
  the functor $\Psi_X \colon \DCAT(\QCOH(X)) \to \DCAT(X)$ from the
  unbounded derived category of quasi-coherent
  $\Orb_X$-modules to the unbounded derived category of
  $\Orb_X$-modules is fully faithful with image $\DQCOH(X)$.
\end{enumerate}
In this generality, \ref{EI:gen} was established in
\cite{MR1996800}---also see \cite{MR1308405}; and \ref{EI:eq} is
essentially due to \cite{MR1422312} (separated case) and
\cite{MR0222093} (noetherian case)---also see \cite[Tags \spref{08DB}
\& \spref{09T4}]{stacks-project}.

Now let $X$ be a quasi-compact and quasi-separated algebraic stack. In
reasonable situations (e.g., affine diagonal), there is a bounded
below equivalence $\DCAT^+(\QCOH(X)) \simeq \DQCOH^+(X)$
\cite[Thm.~C.1]{hallj_neeman_dary_no_compacts}. In these reasonable
situations, if \ref{EI:gen} holds, then so does \ref{EI:eq}
\cite[Thm.~1.2]{hallj_neeman_dary_no_compacts}. For many stacks in
positive characteristic, however  (e.g., $X=B\Ga$), the functor
$\DCAT(\QCOH(X)) \to \DQCOH(X)$ fails to be full and $\DQCOH(X)$ is not
compactly generated \cite[Thms.~1.1 \& 1.3]{hallj_neeman_dary_no_compacts}. 

Local quotient descriptions of algebraic stacks are a natural way to
establish \ref{EI:gen} (and consequently \ref{EI:eq}) for algebraic
stacks. Such descriptions are not easy to come by. Indeed, they
usually rely on Sumihiro-type results for quotient stacks
\cite{MR0337963,MR3329192} or abstract criteria such as
\cite{AHR_lunafield,etale_local_stacks}; also see
\cite[\S2.2.4]{MR2717173}.

The goal of this article is to establish some useful instances of
\ref{EI:gen} and \ref{EI:eq} that do not rely on the existence of
local quotient descriptions. In the smooth case, for example, we can
prove the following result.
\begin{maintheorem}\label{MT:smooth}
  Let $k$ be a field of characteristic $0$. Let $X$ be an algebraic
  stack that is quasi-compact with quasi-affine diagonal. If $X$ is
  smooth over $\spec k$, then 
  \begin{enumerate}
  \item $\DCAT(\QCOH(X)) \simeq \DQCOH(X)$;
  \item $\DCAT^b(\COH(X)) \simeq \PERF(X) \simeq \DQCOH(X)^c$; and
  \item $\DQCOH(X)$ is compactly generated by perfect complexes.
  \end{enumerate}
\end{maintheorem}
We have not seen Theorem \ref{MT:smooth} in the literature before,
though suspect it to be folklore to experts (also see Theorem
\ref{T:reg-conc} for refinements).  
We also develop a new method to prove \ref{EI:eq}, that avoids
\ref{EI:gen}, and use it to establish the following.
\begin{maintheorem}\label{MT:split-presentation}
  Let $k$ be a field of characteristic $0$. Let $X$ be a noetherian algebraic stack over $\spec
  k$. If $X$ has quasi-affine diagonal, then $\DCAT(\QCOH(X)) \simeq \DQCOH(X)$. 
\end{maintheorem}
A particularly useful Corollary of Theorem \ref{MT:split-presentation}
is the following result. This was previously only known under
\'etale-local linearizability of the action (e.g.,
\cite{MR0337963,MR3329192,hallj_neeman_dary_no_compacts}).
\begin{corollary*}
  Let $k$ be a field of characteristic $0$. If $U$ is a noetherian Deligne--Mumford stack with an
  action of an affine algebraic group $G$ over $\spec k$, then
  \[
    \DCAT(\QCOH^G(U)) \simeq \DQCOH([U/G]).
  \]
\end{corollary*}

Theorem \ref{MT:split-presentation} admits various generalizations
(see \S\ref{S:applications}), which are essentially of maximal
generality. For example, we can also prove the following
non-noetherian and mixed characteristic variant.
\begin{maintheorem}\label{MT:aff-diag}
  Let $X$ be a quasi-compact algebraic stack with affine diagonal. If
  $X$ has finite cohomological dimension, then
  $\DCAT(\QCOH(X)) \simeq \DQCOH(X)$.
\end{maintheorem}
Recall that a quasi-compact and quasi-separated algebraic stack $X$
has \emph{cohomological dimension $\leq d$} if $\shfcoho^i(X,F) = 0$
for all $i>d$ and quasi-coherent sheaves $F$ (see
\cite[\S2]{perfect_complexes_stacks} and \cite[Thms.~B \&
C]{hallj_dary_alg_groups_classifying}). Algebraic spaces and algebraic
$\Q$-stacks with affine stabilizers have finite cohomological
dimension. In positive characteristic, having finite cohomological
dimension and affine stabilizers is equivalent to having linearly
reductive stabilizers. In mixed characteristic it is more complicated,
but these criteria persist for stacks with finitely presented inertia
(e.g., noetherian). Some of these subtleties are discussed in detail
in \cite[App.~A]{etale_local_stacks}.

We will give two proofs of Theorem \ref{MT:smooth}. The first proof
will be a direct calculation. The second proof will employ derived
functors of small products in Grothendieck abelian categories. This
approach will also be used to establish Theorem
\ref{MT:split-presentation}. This idea for schemes goes back to
\cite[App.~A]{MR1647519} and for stacks to
\cite{MR2717173,2009arXiv0902.4016H,MR2570954}.

In order to understand derived functors of small products for
algebraic stacks, we use the \emph{descendable} morphisms of Mathew
\cite{MR3459022} (see \S\ref{S:desc-non-aff}). A key technique that we
develop is cohomological calculations of $\SRHom_{\Orb_X}(M,N)$, where
$M$ and $N$ are quasi-coherent $\Orb_X$-modules, and $M$ is
\emph{countably} presented. In particular, $\SRHom_{\Orb_X}(M,N)$ is
not necessarily quasi-coherent. These calculations are new, even for
varieties of finite type over a field.

But we also provide a direct proof of the Corollary above without this
general machinery (see \S\ref{S:direct-corollary}). We also prove
a comparison between quasi-coherent and lisse-\'etale cohomology
(Theorem \ref{T:pushforwardagree}), which generalizes a well-known
result on the bounded below derived category (e.g.,
\cite[Prop.~2.1]{hallj_neeman_dary_no_compacts}).
\subsection*{Acknowledgements}
It is a pleasure to thank David Rydh, Amnon Neeman, Ben Antieau, and Ben Lim
for a number of insightful comments and suggestions.
\section{Direct proof of Theorem \ref{MT:smooth}}
An algebraic stack $X$ is \emph{affine-pointed} if for every field
$k$, every morphism $x\colon \spec k \to X$ is affine. If $X$ has
quasi-affine diagonal, then it is affine-pointed
\cite[Lem.~4.7]{hallj_dary_coherent_tannakian_duality}. Surprisingly, it is an open
question whether affine stabilizers implies affine-pointed
\cite[Qstn.~1.8]{hallj_dary_coherent_tannakian_duality}.

We will prove the following mild refinement of Theorem
\ref{MT:smooth}.
\begin{theorem}\label{T:reg-conc}
  Let $X$ be a noetherian and affine-pointed algebraic stack of finite
  cohomological dimension. If $X$ is locally regular of finite Krull
  dimension, then
  \begin{enumerate}
  \item \label{TI:reg-conc:equiv} $\DCAT(\QCOH(X)) \simeq \DQCOH(X)$;
  \item \label{TI:reg-conc:comp} $\DCAT^b(\COH(X)) \simeq \PERF(X) \simeq \DQCOH(X)^c$; 
  \item \label{TI:reg-conc:gen} $\DQCOH(X)$ is compactly generated by
    perfect complexes; and
  \item \label{TI:reg-conc:fd} There exists an integer $n$ such that for all $M\in \COH(X)$, $N\in \DQCOH(X)$, and $r \in \Z$ there is a quasi-isomorphism:
    \[
      \trunc{\geq -r}\RHom_{\Orb_X}(M,N) \simeq \trunc{\geq
        -r}\RHom_{\Orb_X}(M,\trunc{\geq -r-n}N).
    \]
  \end{enumerate}
\end{theorem}
\begin{proof}
  We first prove \itemref{TI:reg-conc:fd}. Now since $X$ is locally
  regular of finite Krull dimension, there exists an integer $d\geq 0$
  such that for all $M\in \COH(X)$,
  $\SRHom_{\Orb_X}(M,\Orb_X) \in \DCAT^{[0,d]}_{\COH}(X) \subseteq
  \PERF(X)$. Also, $X$ has finite cohomological dimension, so there is an integer $t\geq 0$ such that
  \[
    \trunc{\geq -r}\RDERF \Gamma(X,L) \simeq \trunc{\geq -r}\RDERF \Gamma(X,\trunc{\geq -r-t'}L)
  \]
  for all integers $r$, $t'\geq t$, and $L \in \DQCOH(X)$
  \cite[Thm.~2.6(i)]{perfect_complexes_stacks}. Hence,
  \begin{align*}
    \trunc{\geq -r}\RHom_{\Orb_X}(M,N) &\simeq \trunc{\geq -r}\RDERF \Gamma(X,\RDERF\SHom_{\Orb_X}(M,N))\\
                                       &\simeq \trunc{\geq -r}\RDERF \Gamma(X,\trunc{\geq -r -t}(\SRHom_{\Orb_X}(M,\Orb_X) \tensor^{\LDERF}_{\Orb_X} N))\\
                                       &\simeq \trunc{\geq -r}\RDERF \Gamma(X,\trunc{\geq -r -t}(\SRHom_{\Orb_X}(M,\Orb_X) \tensor^{\LDERF}_{\Orb_X} \trunc{\geq -r -t -d}N))\\
                                       &\simeq \trunc{\geq -r}\RDERF \Gamma(X,\SRHom_{\Orb_X}(M,\Orb_X) \tensor^{\LDERF}_{\Orb_X} \trunc{\geq -r -t -d}N)\\
    &\simeq \trunc{\geq -r}\RHom_{\Orb_X}(M,\trunc{\geq -r -t -d}N).
  \end{align*}
  Taking $n=t+d$ gives the claim.
  
  Since $X$ is has finite cohomological dimension,
  $\PERF(X) \simeq \DQCOH(X)^c$
  \cite[Rem.~4.6]{perfect_complexes_stacks}. Also $X$ is locally
  regular, so $\DCAT^b_{\COH}(X)\simeq \PERF(X)$; hence, if
  $M\in \COH(X)$, then $M[n] \in \DQCOH(X)^c$ for all $n\in \Z$. By
  \cite[Thm.~C.1]{hallj_neeman_dary_no_compacts},
  $\DCAT^b(\QCOH(X)) \simeq \DQCOH^b(X)$. In particular,
  $\DCAT^b_{\COH(X)}(\QCOH(X)) \simeq \DCAT^b_{\COH}(X)$. The
  equivalence $\DCAT^b(\COH(X)) \simeq \DCAT^b_{\COH(X)}(\QCOH(X))$
  follows from a standard argument. Putting these all together, we
  have $\DCAT^b(\COH(X)) \simeq \PERF(X) \simeq \DQCOH(X)^c$. This
  proves \itemref{TI:reg-conc:comp}.

  We will now prove that the set $\COH(X)$ compactly generates
  $\DQCOH(X)$. This will establish \itemref{TI:reg-conc:gen}. Let
  $N\in \DQCOH(X)$ and assume that $\COHO{-r}(N) \neq 0$.  Now
  $\trunc{\geq -r-n}N \in \DQCOH^+(X)$, so it is quasi-isomorphic to a
  bounded below complex
  $Q^\bullet = (\cdots \to Q^d \xrightarrow{\partial^d} Q^{d+1} \to
  \cdots )$ of quasi-coherent sheaves on $X$
  \cite[Thm.~C.1]{hallj_neeman_dary_no_compacts}. By assumption,
  $\COHO{-r}(Q^\bullet) \neq 0$, so there is a non-zero surjection of
  quasi-coherent sheaves
  $\ker (\partial^{-r}) \twoheadrightarrow \COHO{-r}(Q^\bullet)$. Now
  choose a coherent $\Orb_X$-module $M$ together with a morphism
  $M \to \ker (\partial^{-r})$ such that the composition
  $M \to \ker (\partial^{-r}) \to \COHO{-r}(Q^\bullet)$ is non-zero.
  Indeed, every quasi-coherent sheaf on $X$ is a direct limit of its
  coherent subsheaves \cite[Prop.~15.4]{MR1771927} (the separated
  diagonal hypothesis can be safely ignored). It follows immediately
  that there is a non-zero morphism $M[r] \to \trunc{\geq -r-n}N$ in
  $\DQCOH(X)$. By \itemref{TI:reg-conc:fd}, we have a non-zero
  morphism $M[r] \to N$ in $\DQCOH(X)$. The claim follows.

  Finally, \cite[Thm.~1.2]{hallj_neeman_dary_no_compacts} informs us
  that \itemref{TI:reg-conc:gen} implies \itemref{TI:reg-conc:equiv}.
\end{proof}
\begin{proof}[Proof of Theorem \ref{MT:smooth}]
  Since $X$ has quasi-affine diagonal, it is affine pointed
  \cite[Lem.~4.7]{hallj_dary_coherent_tannakian_duality}; in
  particular, it has affine stabilizers. Since $k$ has characteristic
  $0$, it now follows that $X$ has finite cohomological dimension \cite[Thms.~B \&
  C]{hallj_dary_alg_groups_classifying}. Now $X$ is smooth and of
  finite type, so there is a smooth covering $p \colon \spec A \to X$,
  where $\spec A$ is smooth over $\spec k$. It follows immediately
  that $\spec A$ is regular of finite Krull dimension. The result now
  follows from Theorem \ref{T:reg-conc}.
\end{proof}
\section{Derived functors of small products in abelian categories}
Let $\mathcal{A}$ be an abelian category. Let $\Lambda$ be a
set. Suppose that $\Lambda$-indexed products exist in
$\mathcal{A}$. It is convenient to view the product as a functor
\[
  \prod_{\lambda\in \Lambda} \colon \mathcal{A}^\Lambda \to \mathcal{A}
\]
The product category $\mathcal{A}^\Lambda$ is naturally abelian and
the functor $\prod_{\lambda\in \Lambda}$ is readily seen to be
left-exact.

If $\mathcal{A}$ has enough injectives, then so too does the product
category $\mathcal{A}^\Lambda$. Hence, we can consider the
right-derived functors of $\prod_{\lambda \in \Lambda}$, which we
denote via $\prod_{\lambda\in \Lambda}^{(p)}$. We say that
$\mathcal{A}$ satisfies AB$4^*n(\Lambda)$ if
$\prod_{\lambda\in \Lambda}^{(p)} \equiv 0$ for all $p>n$. If
AB$4^*n(\Lambda)$ is satisfied for all sets $\Lambda$, then we say
that $\mathcal{A}$ satisfies AB$4^*n$. This notion was introduced and
studied in \cite{MR2197371}. Just as in the recent work
\cite{antieau_uniqueness_dg}, we will be principally concerned with
the condition AB$4^*n(\omega)$. A number of results related to those
of this section are also established in \cite{2009arXiv0902.4016H}.
\begin{example}\label{E:groth_ab_prod}
  It is a well-known theorem that Grothendieck abelian categories
  admit small products \cite[Tag \spref{07D8}]{stacks-project} and
  have enough injectives. We can also consider the unbounded derived
  category $\DCAT(\mathcal{A})$ of $\mathcal{A}$. Since $\mathcal{A}$
  is Grothendieck abelian, $\DCAT(\mathcal{A})$ also has small
  products. These are formed by taking termwise products of
  K-injective resolutions \cite[Tag
  \spref{07D9}]{stacks-project}. Let $\Lambda$ be a set. If
  $\{A_\lambda\}_{\lambda\in \Lambda}$ is a set of objects of
  $\mathcal{A}$, then
  \[
    \COHO{p}\left(\prod_{\lambda\in \Lambda} A_\lambda[0]\right) \cong
    \begin{cases}
      0 & p<0; \\
      \prod_{\lambda\in \Lambda}^{(p)} A_\lambda & p\geq 0.
    \end{cases}
  \]
  Furthermore, let $n\geq 0$. Assume that $\mathcal{A}$ satisfies
  AB$4^*n(\Lambda)$.  If $\{N_\lambda \}_{\lambda \in \Lambda}$ is a
  set of objects of $\DCAT(\mathcal{A})$, then for every $r\in \Z$ we
  have
  \begin{equation}
  \trunc{>r}\prod_{\lambda\in \Lambda} N_\lambda \simeq
  \trunc{>r}\prod_{\lambda\in \Lambda} \trunc{>r-n}N_\lambda.\label{eq:products}
\end{equation}
This is similar to \cite[Tag \spref{07K7}]{stacks-project}, but we
reproduce the argument here. Also see \cite[Prop.~8.14]{antieau_uniqueness_dg}. Let $N_\lambda \to I_\lambda^{\bullet}$
be a quasi-isomorphism to a K-injective complex of
$\mathcal{A}$-injectives. Consider the exact sequence of complexes:
  \[
    \xymatrix{0 \ar[r] & L_\lambda^{r}[-r+n] \ar[r] &
      \sigma^{>r-n-1}I_\lambda^{\bullet} \ar[r] &
      \trunc{>r-n-1}I_\lambda^{\bullet} \ar[r] & 0,}
  \]
  where $\sigma$ denotes the brutal truncation and
  $L_\lambda^r \in \mathcal{A}$. Taking products in
  $\DCAT(\mathcal{A})$ we obtain a distinguished triangle:
  \[
    \xymatrix@C-1pc{ \displaystyle{\prod_{\lambda\in \Lambda}}L_\lambda^r[-r+n]
      \ar[r] & \displaystyle{\prod_{\lambda\in \Lambda}}
      \sigma^{>r-n-1}I_\lambda^{\bullet} \ar[r] & \displaystyle{\prod_{\lambda\in
        \Lambda}} \trunc{>r-n-1}I_\lambda^{\bullet} \ar[r] &
      \displaystyle{\prod_{\lambda\in \Lambda}}L_\lambda^r[-r+n+1]}
  \]
  Now
  \[
    \prod_{\lambda\in \Lambda} \sigma^{>r-n-1}I_\lambda^{\bullet}
    \simeq \sigma^{>r-n-1}\prod_{\lambda\in \Lambda}
    I_\lambda^{\bullet}.
  \]
  This implies that 
  \[
    \trunc{>r-n}\prod_{\lambda\in
      \Lambda} \sigma^{>r-n-1}I_\lambda^{\bullet} \simeq
    \trunc{>r-n}\prod_{\lambda\in \Lambda} N_\lambda.
  \]
  But $\mathcal{A}$ satisfies AB$4^*n(\Lambda)$, so 
  \[
    \trunc{>r}\prod_{\lambda\in \Lambda}L_\lambda^r[-r+n] \simeq \trunc{>r}\left((\prod_{\lambda\in \Lambda} L_\lambda^r[0])[-r+n]\right) \simeq 0.
  \]
  It follows that
  \[
    \trunc{>r}\prod_{\lambda\in \Lambda} N_\lambda \simeq
    \trunc{>r}\prod_{\lambda\in \Lambda}\trunc{>r-n-1}N_\lambda.
  \]
  We also have a distinguished triangle in $\DCAT(\mathcal{A})$:
  \[
    \xymatrix@-1.5pc{\displaystyle{\prod_{\lambda\in \Lambda}} \COHO{r-n}(N_\lambda)[n-r] \ar[r]
      & \displaystyle{\prod_{\lambda\in \Lambda} \trunc{>r-n}} N_\lambda \ar[r] &
      \displaystyle{\prod_{\lambda\in \Lambda}} \trunc{>r-n-1}N_\lambda \ar[r] &
      \displaystyle{\prod_{\lambda\in \Lambda}} \COHO{r-n}(N_\lambda)[n-r+1].}
  \]
  As before,
  $\trunc{>r}{\prod_{\lambda\in \Lambda}} \COHO{r-n}(N_\lambda)[n-r]
  \simeq 0$. It follows now that
 \[
    \trunc{>r}\prod_{\lambda\in \Lambda} N_\lambda \simeq
    \trunc{>r}\prod_{\lambda\in \Lambda}\trunc{>r-n}N_\lambda.
  \] 
\end{example}
\begin{example}
  Let $X$ be an algebraic stack. Then the abelian category $\QCOH(X)$
  is Grothendieck abelian. If $X$ is noetherian, this is reasonably
  straightforward, and follows immediately from
  \cite[Prop.~15.4]{MR1771927}. In general, this is \cite[Tag
  \spref{0781}]{stacks-project}.
\end{example}
\begin{example}\label{E:site-stalk-prod}
  Let $\mathcal{C}$ be a site. Let $\Orb_{\mathcal{C}}$ be a sheaf of
  rings on $\mathcal{C}$. Then $\MOD(\mathcal{C},\Orb_{\mathcal{C}})$
  is Grothendieck abelian. Let $u \colon \mathcal{C} \to \SETS$ be
  a point of $\mathcal{C}$. If $\{C_\lambda\}_{\lambda\in \Lambda}$ is
  a set of sheaves of $\Orb_{\mathcal{C}}$-modules, then for all
  $p\geq 0$ there is a natural isomorphism
  \[
    \Bigl(\prod_{\lambda \in \Lambda}^{(p)} C_\lambda\Bigr)_u \simeq
    \varinjlim_{(V,v)} \prod_{\lambda \in \Lambda}
    \shfcoho^p(V,C_\lambda),
  \]
  where the limit is over a cofinal system of neighbourhoods $(V,v)$
  of $u$. This follows by making the natural modifications to the
  argument given in \cite[Prop.~1.6]{MR2197371}.
\end{example}
\begin{example}\label{E:adjoints-prod}
  Let $F \colon \mathcal{A} \leftrightarrows \mathcal{A}' \colon G$ be
  an adjoint and exact pair of functors between Grothendieck abelian
  categories. Since $F$ is exact, 
  $\RDERF G \colon \DCAT(\mathcal{A}') \to \DCAT(\mathcal{A})$ has a
  left adjoint, so preserves products. Hence, if
  $\{M_\lambda\}_{\lambda\in \Lambda}$ is a set objects of
  $\mathcal{A}'$, then
  \[
    \RDERF G(\prod_{\lambda\in \Lambda} M_\lambda[0]) \simeq
    \prod_{\lambda \in \Lambda} \RDERF G(M_\lambda[0]).
  \]
  But $\prod_{\lambda\in \Lambda}$ preserves injectives, 
  so 
  $\prod^{(p)}_{\lambda\in \Lambda} G(M_\lambda) \simeq
  G(\prod^{(p)}_{\lambda\in \Lambda} M_\lambda)$. In particular,
  \begin{enumerate}
  \item\label{EI:adjoints-prod:contra} if
    $\prod_{\lambda\in \Lambda}^{(n)} G(M_\lambda) \neq 0$ for some
    $n$, then $\mathcal{A}'$ does not satisfy AB$4^*n(\Lambda)$; and
  \item\label{EI:adjoints-prod:cons} if $G$ is conservative and
    $\mathcal{A}$ satisfies AB4${}^*n(\Lambda)$ for some $n$, then
    the same is true of $\mathcal{A}'$.
  \end{enumerate}
\end{example}
We now give some key examples. While we have not seen these in the
literature before, we expect at least some of them to have been known
in some cases to experts.
\begin{example}
\label{E:ab4n-reg}
Let $X$ be an affine-pointed algebraic stack of finite cohomological dimension. If $X$
is locally regular of finite Krull dimension, then $\QCOH(X)$
satisfies AB$4^*n$ for some $n>0$. To see this, we will use the criterion of
\cite[Thm.~1.3(c')]{MR2197371}. It is thus sufficient to exhibit an
$n$ such that $\Ext^{i}_{\QCOH(X)}(M,-)$ for all $i>n$ and
$M \in \COH(X)$. Observe that
$\DCAT^+(\QCOH(X)) \simeq \DQCOH^+(X)$
\cite[Thm.~C.1]{hallj_neeman_dary_no_compacts}; thus, if
$M \in \COH(X)$ and $N\in \QCOH(X)$, then
  \begin{align*}
    \Ext^i_{\QCOH(X)}(M,N) &\simeq \COHO{i}(\RHom_{\Orb_X}(M,N)) \simeq \COHO{0}(\RHom_{\Orb_X}(M,N[i])). 
  \end{align*}
  Now take $n$ as in Theorem \ref{T:reg-conc}\itemref{TI:reg-conc:fd}. 
\end{example}
\begin{example}\label{E:ab4n-res}
  Let $X$ be a quasi-compact algebraic stack with affine diagonal. If
  $X$ has the compact resolution property \cite[\S
  7]{perfect_complexes_stacks}, then $\QCOH(X)$ satisfies
  AB$4^*n$ for some $n>0$. For example, algebraic stacks of finite cohomological dimension with the
  resolution property (i.e., every finite type quasi-coherent sheaf is
  a quotient of a finite rank vector bundle) have the compact
  resolution property.  The argument given in Example \ref{E:ab4n-reg}
  works with only minor changes, see \cite[Lem.~7.6]{perfect_complexes_stacks}.
\end{example}
\begin{example}\label{E:counter}
  Let $k$ be a field of characteristic $p>0$. Then
  $\QCOH(B\mathbf{G}_{a,k})$ does not satisfy AB$4^*n(\omega)$ for any
  $n>0$ \cite{MR2875857}. This holds, more generally, for \emph{poorly
    stabilized} algebraic stacks
  \cite[\S4]{hallj_neeman_dary_no_compacts}. This condition is
  equivalent to having a point of characteristic $p>0$ and a subgroup
  of the stabilizer at that point isomorphic to $\Ga$ (e.g.,
  $B\GL_{3,\F_2}$ but not $B\Gm$).
\end{example}
\begin{example}\label{E:spectral-space-mod}
  If $X=(|X|,\Orb_X)$ is a ringed space, where $|X|$ is spectral
  (e.g., noetherian) of Krull dimension $d$, then $\MOD(X)$ satisfies
  AB$4^*d$. By Example \ref{E:site-stalk-prod}, it is sufficient to
  prove that for every $x\in |X|$ there is a cofinal system of
  neighbourhoods $V$ of $x$ with $\shfcoho^{q}(V,F)=0$ for all
  sheaves of $\Orb_X$-modules $F$. Now $|X|$ has Krull dimension $d$,
  so $V$ has Krull dimension $\leq d$ for every open
  $V \subseteq |X|$. It follows that $\shfcoho^q(V,F) = 0$ for all
  $q>d$ \cite{MR1179103} (also see \cite[Tag
  \spref{0A3G}]{stacks-project}) for every quasi-compact open
  $V \subseteq |X|$. But $|X|$ is spectral, so the quasi-compact open
  subsets form a basis, and we have the claim.
\end{example}
\begin{example}\label{MT:dm-mod-abn}
  If $X$ is a quasi-compact and quasi-separated Deligne--Mumford stack
  of Krull dimension $d$ (e.g., noetherian), then $\MOD(X_{\et})$
  satisfies AB$4^*d$. To see this, choose an \'etale cover
  $p\colon U \to X$ by a quasi-compact and quasi-separated scheme $U$
  of Krull dimension $d$. Then we have an adjoint pair of functors:
  \[
    p_! \colon \MOD(U_{\et}) \leftrightarrows \MOD(X_{\et}) \colon p^{-1},
  \]
  where both $p_!$ and $p^{-1}$ are exact, and $p^{-1}$ is
  conservative. By Example \ref{E:adjoints-prod}, we are reduced to
  the situation where $X=U$ is a scheme with $|X|$ spectral of Krull
  dimension $d$. Now argue just as in Example
  \ref{E:spectral-space-mod}, but this time use quasi-compact \'etale
  morphisms $V \to X$ as the basis.  
\end{example}
\begin{example}\label{E:ff_desc}
  Let $X$ be a quasi-compact and quasi-separated algebraic stack and
  let $\Lambda$ be a set. Let $p\colon X' \to X$ be a flat and affine morphism.
  \begin{enumerate}
  \item \label{EI:ff_desc:affine} If $\QCOH(X)$ satisfies AB$4^*n(\Lambda)$, then $\QCOH(X')$ satisfies AB$4^*n(\Lambda)$.
  \item \label{EI:ff_desc:finite} If $\QCOH(X')$ satisfies AB$4^*n(\Lambda)$ and $p$ is finite, finitely presented, and surjective, then $\QCOH(X)$ satifies AB$4^*n(\Lambda)$. 
  \end{enumerate}
  For \itemref{EI:ff_desc:affine}: we have the adjoint pair
  $p^* \colon \QCOH(X) \leftrightarrows \QCOH(X') \colon p_*$, where
  $p^*$ is exact and $p_*$ is exact and conservative ($p$ is
  affine). By Example \ref{E:adjoints-prod}, the claim follows.
  
  Similarly, for \itemref{EI:ff_desc:finite}: we have an adjoint pair
  $p_* \colon \QCOH(X') \leftrightarrows \QCOH(X) \colon p^\times$,
  where $p_*$ is exact and $p^\times$ is exact and conservative
  \cite[Cor.~4.15]{perfect_complexes_stacks}. The claim follows from
  Example \ref{E:adjoints-prod} again.
\end{example}
\begin{example}\label{E:open}
A variant of Example \ref{E:ff_desc}\itemref{EI:ff_desc:affine} is the following. Let $j\colon U \subseteq X$ be a quasi-compact open immersion of algebraic stacks. Assume that $X$ is quasi-compact with affine diagonal or noetherian and affine-pointed. If $\QCOH(X)$ is AB$4^*n(\Lambda)$, then so is $\QCOH(U)$. Indeed, let $\{N_\lambda\}_{\lambda \in \Lambda}$ be a set of quasi-coherent sheaves on $U$. By \cite[Prop.~2.1]{hallj_neeman_dary_no_compacts} it follows immediately that there is a quasi-isomorphism in $\DCAT(\QCOH(U))$:
\[
  \prod_{\lambda \in \Lambda} N_\lambda \simeq j^*\RDERF
  j_{\QCOH,*}(\prod_{\lambda\in \Lambda} N_\lambda) \simeq
  j^*(\prod_{\lambda \in \Lambda} \RDERF j_{\QCOH,*} N_\lambda).
\]
The right hand side is bounded above, so the claim follows.
\end{example}
\begin{example}\label{E:mv_desc}
  Consider a cartesian square of algebraic stacks:
  \[
    \xymatrix{U' \ar[r]^{j'} \ar[d]_{f_U} & X' \ar[d]^f \\ U \ar[r]^j & X}
  \]
  that are all quasi-compact with affine diagonal or noetherian and affine-pointed. Assume that 
  \begin{enumerate}
  \item $j$ is a quasi-compact open immersion with finitely presented complement $i\colon Z \hookrightarrow X$; and
  \item $f$ is concentrated, flat, and $f_Z \colon X'\times_X Z \to Z$ is an isomorphism. 
  \end{enumerate}
	 That is, the square is a flat Mayer--Vietoris square, in the sense of \cite[Defn.~1.2]{mayer-vietoris}. This is satisfied, for example, when $f$ is a representable \'etale neigbourhood of $Z$ (i.e., $f$ is \'etale and $f_Z$ is an isomorphism).  We now claim that if $X'$ and $U'$ satisfy AB$4^*n(\Lambda)$ for some $n$, then $X$ satisfies AB$4^*m(\Lambda)$ for some $m$. To see this, we let $\{ M_\lambda \}_{\lambda \in \Lambda}$ be a set of quasi-coherent sheaves on $X$. We may form the Mayer--Vietoris triangle (see point (i) of the proof of \cite[Thm.~4.4]{mayer-vietoris}) in $\DQCOH^+(X)$:
	 \[
	 	\xymatrix{M_\lambda \ar[r] & \RDERF j_{*}j^*M_\lambda \oplus \RDERF f_*f^*M_\lambda \ar[r] & \RDERF f_*f^*\RDERF j_*j^*M_{\lambda} \ar[r] & M_{\lambda}[1].}
	 \]	
	 By \cite[Prop.~2.1]{hallj_neeman_dary_no_compacts}, this induces the following distinguished triangle in $\DCAT(\QCOH(X))$:
	 \[
	 	\xymatrix{M_\lambda \ar[r] & \RDERF j_{\QCOH,*}j^*M_\lambda \oplus \RDERF f_{\QCOH,*}f^*M_\lambda \ar[r] & \RDERF f_{\QCOH,*}f^*\RDERF j_{\QCOH,*}j^*M_{\lambda} \ar[r] & M_{\lambda}[1].}
	 \]	
	 We now take the product over $\lambda \in \Lambda$ in $\DCAT(\QCOH(X))$ to obtain the following triangle:
	 \[
	 \xymatrix@C-1.5pc{\displaystyle{\prod_{\lambda \in \Lambda}} M_\lambda \ar[r] & \displaystyle{\prod_{\lambda \in \Lambda} (\RDERF j_{\QCOH,*}j^*M_\lambda \oplus \RDERF f_{\QCOH,*}f^*M_\lambda)} \ar[r] & \displaystyle{\prod_{\lambda \in \Lambda} \RDERF f_{\QCOH,*}f^*\RDERF j_{\QCOH,*}j^*M_{\lambda} }\ar[r] & \displaystyle{\prod_{\lambda \in \Lambda} M_{\lambda}[1].}}
	 \]	
	 Since pushforwards commute with products, we see that $\prod_{\lambda \in \Lambda} M_\lambda$ belongs to a triangle whose other terms are 
	 \begin{align*}
	 \displaystyle{\prod_{\lambda \in \Lambda} (\RDERF j_{\QCOH,*}j^*M_\lambda \oplus \RDERF f_{\QCOH,*}f^*M_\lambda)} &\simeq \displaystyle{\RDERF j_{\QCOH,*}\left(\prod_{\lambda \in \Lambda}j^*M_\lambda\right) \oplus \RDERF f_{\QCOH,*}\left(\prod_{\lambda \in \Lambda}f^*M_\lambda\right)}
\end{align*}	 	
and \begin{align*}
\displaystyle{\prod_{\lambda \in \Lambda} \RDERF f_{\QCOH,*}f^*\RDERF j_{\QCOH,*}j^*M_{\lambda} } &\simeq \displaystyle{ \RDERF f_{\QCOH,*}\prod_{\lambda \in \Lambda}f^*\RDERF j_{\QCOH,*}j^*M_{\lambda} }.
\end{align*}
Since $f$ and $j$ are concentrated, they have cohomological dimension
bounded by some $r$. It follows that both of the above objects of
$\DCAT(\QCOH(X))$ have cohomological support bounded by $r+n$. It
follows immediately that $\prod_{\lambda \in \Lambda} M_\lambda$ has
cohomological support bounded by $r+n$. Hence, $\QCOH(X)$ satisfies
AB$4^*(r+n)(\Lambda)$.
\end{example}
Combining the examples above with the \'etale devissage results of
\cite{etale_dev_add}, one easily obtains the following.
\begin{proposition}\label{P:qfdiag_ab4starn}
  Let $X$ be an algebraic stack. Assume 
  \begin{enumerate}
  \item $X$ has quasi-finite and affine diagonal; or
  \item $X$ is noetherian with quasi-finite and separated diagonal; or 
  \item $X$ is a noetherian Deligne--Mumford $\Q$-stack.
  \end{enumerate}
  Then $\QCOH(X)$ satisfies \emph{AB}$4^*n$ for some $n>0$. 
\end{proposition}
We conclude this section with a derived variant of Example
\ref{E:adjoints-prod}.
\begin{example}\label{E:adjoints-prod-refined}
  Let $\mathcal{A}$ be a Grothendieck abelian and let
  $\mathcal{M} \subseteq \mathcal{A}$ be a weak Serre subcategory that
  is closed under coproducts
  \cite[App.~A]{hallj_neeman_dary_no_compacts}. Consider the functor
  \[
    \Psi \colon \DCAT(\mathcal{M}) \to \DCAT_{\mathcal{M}}(\mathcal{A}).
  \]
  Assume:
  \begin{enumerate}
  \item $\mathcal{M}$ is Grothendieck abelian; and
  \item
    $\Psi^+ \colon \DCAT^+(\mathcal{M}) \to
    \DCAT^+_{\mathcal{M}}(\mathcal{A})$ is an equivalence.
  \end{enumerate}
  Then $\DCAT_{\mathcal{M}}(\mathcal{A})$ is well-generated
  \cite[Thm.~A.3]{hallj_neeman_dary_no_compacts}, so admits small
  products. Also, $\Psi$ preserves coproducts, so the functor
  $\Psi \colon \DCAT(\mathcal{M}) \to
  \DCAT_{\mathcal{M}}(\mathcal{A})$ admits a right adjoint
  $\Phi \colon \DCAT_{\mathcal{M}}(\mathcal{A}) \to
  \DCAT(\mathcal{M})$ \cite[Prop.~1.20]{MR1812507}. In particular,
  $\Phi$ preserves products. These conditions are satisfied for the
  inclusion $\QCOH(X) \subseteq \MOD(X)$, where $X$ is an algebraic
  stack that is quasi-compact with affine diagonal or is noetherian and
  affine-pointed \cite[Thm.~C.1]{hallj_neeman_dary_no_compacts}.

  Now since $\Phi$ is right adjoint to a (right) $t$-exact functor,
  then $\Phi$ is left $t$-exact. In particular, $\Phi$ restricts to a
  functor
  $\Phi^+ \colon \DCAT^+_{\mathcal{M}}(\mathcal{A}) \to
  \DCAT^+(\mathcal{M})$. Let $M \in \DCAT(\mathcal{M})$ and form the
  distinguished triangle:
  \[
    \xymatrix{M \ar[r] & \Phi \Psi(M) \ar[r] & C(M) \ar[r] & M[1].}
  \]
  Now assume that $M$ is bounded below. Then so too is $C(M)$ and applying
  $\RHom_{\mathcal{M}}(C(M),-)$ to this triangle produces the triangle
  \[
    \xymatrix@C-1.5pc{\RHom_{\mathcal{M}}(C(M),M) \ar[r] & \RHom_{\mathcal{M}}(C(M),\Phi\Psi(M)) \ar[r] & \RHom_{\mathcal{M}}(C(M),C(M)) \ar[r] & \RHom_{\mathcal{M}}(C(M),M)[1].}
  \]
  By full faithfulness of $\Psi^+$, the composition
  \[
    \RHom_{\mathcal{M}}(C(M),M) \to \RHom_{\mathcal{M}}(C(M),\Phi\Psi(M)) \simeq \RHom_{\DCAT_{\mathcal{M}}(\mathcal{A})}(\Psi(C(M)),\Psi(M))
  \]
  is a quasi-isomorphism. Hence, $C(M) = 0$ and $M \to \Phi\Psi(M)$ is
  a quasi-isomorphism for all $M\in \DCAT^+(\mathcal{M})$. To show
  that $\Psi$ and $\Phi$ restrict to adjoint equivalences on
  $\DCAT^+$, it remains to prove that
  $N \in \DCAT^+_{\mathcal{M}}(\mathcal{A})$ implies that
  $\Phi(N) \neq 0$. But this is clear: $N \simeq \Psi(M)$ for some
  $M\in \DCAT^+(\mathcal{M})$ and so
  $0 \simeq \Phi(N) \simeq \Phi\Psi(M) \simeq M$. Hence,
  $N \simeq \Phi(M) \simeq 0$.

  Let $\Lambda$ be a
  set. If $\mathcal{M}$ satisfies AB$4^*n(\Lambda)$ for some $n$, then
  for $-\infty < a\leq b < \infty$ and a subset
  $\{N_\lambda\}_{\lambda \in \Lambda} \subseteq
  \DCAT^{[a,b]}_{\mathcal{M}}(\mathcal{A})$, we have that
  $\prod_{\lambda \in \Lambda}^{\DCAT_{\mathcal{M}}(\mathcal{A})}
  N_\lambda \in \DCAT^{[a,b+n]}_{\mathcal{M}}(\mathcal{A})$.  Indeed,
  $\prod_{\lambda \in \Lambda}^{\DCAT_{\mathcal{M}}(\mathcal{A})}
  N_\lambda\in \DCAT^{+}_{\mathcal{M}}(\mathcal{A})$, so
  \begin{align*}
    \prod_{\lambda \in \Lambda}^{\DCAT_{\mathcal{M}}(\mathcal{A})}
    N_\lambda \simeq  \Psi\Phi\left( \prod_{\lambda \in \Lambda}^{\DCAT_{\mathcal{M}}(\mathcal{A})}
                N_\lambda\right)
              \simeq \Psi\left( \prod_{\lambda \in \Lambda}^{\DCAT({\mathcal{M}})}
                \Phi(N_\lambda)\right)
  \end{align*}
  But $\Phi$ is $t$-exact on $\DCAT^+_{\mathcal{M}}(\mathcal{A})$, so
  $\Phi(N_\lambda) \in \DCAT^{[a,b]}(\mathcal{M})$.  The claim follows
  from Example \ref{E:groth_ab_prod}. Conversely, if
  $\{ M_\lambda\}_{\lambda\in \Lambda} \subseteq \mathcal{M}$ and
  $\trunc{>n}(\prod^{\DCAT_{\mathcal{M}}(\mathcal{A})}_{\lambda\in
    \Lambda} \Psi(M_\lambda[0])) \simeq 0$, then
  $\prod^{(s)}_{\lambda\in \Lambda} M_\lambda = 0$ for all $s>n$ in
  $\mathcal{M}$. In particular, if this is true for all families
  $\{M_\lambda\}_{\lambda\in \Lambda}$, then $\mathcal{M}$ satisfies
  AB$4^*n(\Lambda)$. Indeed,
  \[
    \prod_{\lambda\in \Lambda}^{\DCAT(\mathcal{M})} M_\lambda[0]
    \simeq \prod_{\lambda\in \Lambda}^{\DCAT(\mathcal{M})}
    \Phi\Psi(M_\lambda[0]) \simeq
    \Phi\left(\prod_{\lambda\in
        \Lambda}^{\DCAT_{\mathcal{M}}(\mathcal{A})}\Psi(M_\lambda[0])\right).
  \]
  But
  $\prod_{\lambda\in \Lambda}^{\DCAT_{\mathcal{M}}(\mathcal{A})}\Psi(M_\lambda[0])
  \in \DCAT^{\geq 0}_{\mathcal{M}}(\mathcal{A})$, so $\Phi$ is $t$-exact. The
  claim follows.
\end{example}

\section{Left completeness of $t$-structures on a triangulated category}
Let $\mathcal{T}$ be a triangulated category. Consider an inverse system of objects in $\mathcal{T}$:
\[
\cdots \xrightarrow{\theta_{r+2}}	t_{r+1} \xrightarrow{\theta_{r+1}} t_r \xrightarrow{\theta_r} t_{r-1} \xrightarrow{\theta_{r-1}} \cdots.
\]	
If $\mathcal{T}$ has countable products, then the \emph{homotopy
  limit} of this system is the (unique, up to non-unique isomorphism)
object that fits into the distinguished triangle:
\[
\xymatrix{ \holim{r} t_r \ar[r] & \prod_r t_r \ar[r]^{1-[\mathrm{shift}]} & \prod_r t_r \ar[r] & \holim{r} t_r[1].}
\]
\begin{example}\label{E:ab4nhtyp}
  If $\mathcal{A}$ is a Grothendieck abelian category. Consider an
  inverse system of objects $\{t_r\}$ in $\DCAT(\mathcal{A})$ as
  above. If $\mathcal{A}$ satisfies AB$4^*n(\omega)$ for some $n$,
  then for each $p\in \Z$ we have
  \[
    \trunc{\geq p}\holim{r} t_r \simeq \trunc{\geq p}\holim{r}\trunc{\geq p-n-1}t_r
  \]
  This is immediate from \eqref{eq:products} and the defining triangle
  for $\holim{r}$
\end{example}
For each $x\in \mathcal{T}$ and $p \in \Z$, applying $\RHom(x,-)$ to
this triangle results in the Milnor exact sequence \cite[Tag
\spref{0919}]{stacks-project}:
\begin{equation}
  \xymatrix@C-0.8pc{0 \ar[r] & \varprojlim^1_r \Ext^{p-1}(x,t_r) \ar[r] & \Ext^p(x,\holim{r} t_r) \ar[r] & \varprojlim_r \Ext^p(x,t_r) \ar[r] & 0.}\label{eq:milnor}
\end{equation}
One easily deduces from this that the homotopy limit of a constant system associated to $t \in \mathcal{T}$ (i.e., $t_r=t$ and $\theta_r = \ID{}$ for all $r$) is just $t$ and that the homotopy limit of a general system only depends upon its tail.
 
A very nice and useful consequence of this sequence is the following well-known example.
\begin{example}\label{ex:milnor_ab}
  Let $\mathcal{T}=\DCAT(A)$, where $A$ is a ring. Then applying
  \eqref{eq:milnor} with $x=A$ we obtain an exact sequence of
  $A$-modules for all $p\in \Z$:
  \[
   \xymatrix{0 \ar[r] & \varprojlim^1_r \COHO{p-1}(t_r) \ar[r] & \COHO{p}(\holim{r} t_r) \ar[r] & \varprojlim_r \COHO{p}(t_r) \ar[r] & 0.}
  \]
  It follows that for all $p\in \Z$:
  \[
    \trunc{>p}\holim{r}t_r \simeq \trunc{>p}\holim{r}(\trunc{>p-1}t_r).
  \]
  This also follows from Example \ref{E:ab4nhtyp}. Moreover, if
  $\trunc{>p}\holim{r}t_r \simeq 0$, then
  $\holim{r}\trunc{>p}t_r \simeq 0$. Indeed, the above shows that
  $\trunc{>p+1}\holim{r}(\trunc{>p}t_r) \simeq 0$ and
  $\trunc{<p+1}\holim{r}(\trunc{>p}t_r) \simeq 0$. It remains to check
  what happens in degree $p+1$. But now we have a commutative diagram
  with exact rows:
  \[
    \xymatrix{0 \ar[r] & \varprojlim^1_r \COHO{p}(t_r) \ar[d] \ar[r] & \ar[d] \COHO{p+1}(\holim{r} t_r) \ar[r] & \varprojlim_r \COHO{p+1}(t_r) \ar[r] \ar[d] & 0\\
      0 \ar[r] & \varprojlim^1_r \COHO{p}(\trunc{>p}t_r) \ar[r] &
      \COHO{p+1}(\holim{r} \trunc{>p}t_r) \ar[r] & \varprojlim_r \COHO{p+1}(\trunc{>p}t_r)
      \ar[r] & 0.}
  \]
  Now the right vertical arrow is an isomorphism and bottom left entry
  is zero. By the Snake Lemma, the middle vertical map is
  surjective. But the middle entry in the top row is zero and the
  claim follows.
\end{example}

Now assume that $\mathcal{T}$ has a $t$-structure. We say that
$\mathcal{T}$ is:
\begin{itemize}
\item \emph{left separated} if $t\in \mathcal{T}$ satisfies
  $\trunc{>-r}t \simeq 0$ for all $r\in \Z$, then $t\simeq 0$;
\item \emph{left faithful} if for each
$t\in \mathcal{T}$, any induced morphism:
\[
t \to \holim{r} \trunc{\geq -r}t
\]
is an isomorphism (note that we only need to check it for one such
morphism);
\item \emph{left complete} if it is left separated and given an inverse system
  \[
    \cdots \to t_{r+1} \xrightarrow{\theta_{r+1}} t_r \to \cdots
  \]
  in $\mathcal{T}$, where $\trunc{>-r}t_{r+1} \simeq t_r$ for all
  $r\in \Z$ (i.e., a Postnikov pretower), there exists
  $t\in \mathcal{T}$ and compatible maps $\psi_r \colon t \to t_r$
  such that $\trunc{>-r}t \simeq t_r$ for all $r\in \Z$.
\end{itemize}
We now make some useful remarks.
\begin{remark}
  Certainly, left faithful implies left separated. A brief argument
  also shows that left complete implies left faithful. 
\end{remark}
\begin{remark}
  In the triangulated category literature, what we call
  left faithful is sometimes called left complete.
\end{remark}
\begin{remark}
  If $\mathcal{C}$ is a stable $\infty$-category with a $t$-structure,
  then there is also a notion of its \emph{left completion}
  \cite[Prop.\ 1.2.1.17]{lurie_highalg}. This produces another stable
  $\infty$-category $\hat{\mathcal{C}}$ with a $t$-structure, together
  with a $t$-exact functor
  $f \colon \mathcal{C} \to \hat{\mathcal{C}}$, which induces an
  equivalence
  $\mathcal{C}^{\geq 0} \simeq \hat{\mathcal{C}}^{\geq
    0}$. Explicitly, $\hat{\mathcal{C}}=\lim_n \mathcal{C}^{\geq
    -n}$. This can be viewed as the subcategory of
  $\mathsf{Fun}(\mathrm{N}(\Z),\mathcal{C})$ with objects those
  functors $F \colon N(\Z) \to \mathcal{C}$ such that
  $F(n) \in \mathcal{C}^{\geq n}$ and if $m\leq n$ in $\Z$, then the
  induced map $F(m) \to F(n)$ induces an equivalence
  $\trunc{\geq n}F(m) \simeq F(n)$. Lurie then defines $\mathcal{C}$
  to be \emph{left complete} if $f$ is an equivalence. If
  $\mathcal{C}$ admits countable products, then the functor
  $f \colon \mathcal{C} \to \hat{\mathcal{C}}$ admits a right adjoint
  $g \colon \hat{\mathcal{C}} \to \mathcal{C}$. Explicitly,
  $g(F) = \holim{n} F(n)$. Note that if $c\in \mathcal{C}$, then
  $gf(c) = \holim{n} \trunc{\geq n}c$. Similarly, if
  $F \in \hat{\mathcal{C}}$, then $fg(F)$ is the functor
  $n \mapsto \trunc{\geq n}(\holim{s} F(s))$.  Then $f$ is fully
  faithful (resp.~an equivalence) if and only if the homotopy category
  $\mathrm{h}(\mathcal{C})$, which is triangulated, is left faithful
  (resp.~left complete).  To see this, we note that $f$ is fully
  faithful if and only if the adjunction
  $c \to gf(c)=\holim{n} \trunc{\geq n}c$ is an equivalence for all
  $c\in \mathcal{C}$. Similarly, if $f$ is fully faithful, then it is
  an equivalence if and only if $g$ is conservative. Now it suffices
  to check these equivalences on the homotopy category. Observe that
  the functor $\pi \colon \mathcal{C} \to \mathrm{h}(\mathcal{C})$
  preserves holim.\footnote{It suffices to show that $\pi$ preserves
    countable products. Let $\{y_n\}_{n\in \Z}$ belong to
    $\mathcal{C}$. By assumption, there is a $y\in \mathcal{C}$ that
    represents the functor
    $x\mapsto \prod_{n}\Hom^\ast_{\mathcal{C}}(x,y_n)$ from
    $\mathcal{C}^\opp \to \DCAT(\AB)$. It suffices to prove that
    $\pi(y)$ represents the cohomological functor
    $z\mapsto \prod_n \Hom_{\mathrm{h}(\mathcal{C})}(z,\pi(y_n))$ from
    $\mathrm{h}(\mathcal{C})^\opp \to \AB$. This is clear: the
    relationship here is just taking $\COHO{0}(-)$, which preserves
    products.} In particular, the full faithfulness claim is now
  clear. If $\mathrm{h}(\mathcal{C})$ is left complete, let
  $F \in \hat{\mathcal{C}}$ be such that $g(F) = 0$. Let
  $c_n = \pi(F(n))$; then $0=\pi(g(F)) = \holim{n} c_n$. By left
  completeness, $c_n =0$ for all $n$. The converse is clear.
\end{remark}
\begin{remark}
If $t \in \mathcal{T}$ and $\trunc{<s}t \simeq 0$ for some $s\in \Z$, then any induced morphism
\[
	t \to \holim{r} \trunc{\geq -r}t
\]
is an isomorphism. Indeed, as the homotopy limit of an inverse system only depends upon the tail, we may replace $\{\trunc{\geq -r}t\}_r$ with the constant system $\{ \trunc{\geq s} t\}_r$. The claim follows.
\end{remark}
\begin{lemma}\label{L:left-complete}
  Let $\mathcal{A}$ be a Grothendieck abelian category.
  \begin{enumerate}
  \item \label{LI:left-complete:sep} $\DCAT(\mathcal{A})$ is
    left separated with respect to the standard $t$-structure.
  \item\label{LI:left-complete:complete} If $\mathcal{A}$ satisfies
    \textup{AB}$4^*n(\omega)$, then $\DCAT(\mathcal{A})$ is
    left complete with respect to the standard {$t$-structure}.
\end{enumerate}
\end{lemma}
\begin{proof}
  Claim \itemref{LI:left-complete:sep} is trivial. 
  For claim
  \itemref{LI:left-complete:complete}: by
  \itemref{LI:left-complete:sep}, we know $\DCAT(\mathcal{A})$ is
  left separated. Now consider an inverse system in
  $\DCAT(\mathcal{A})$
  \[
    \cdots \to A_{r+1} \xrightarrow{\theta_{r+1}} A_r \to \cdots,
  \]
  with $\trunc{>-r}A_{r+1} \simeq A_r$. Let $\hat{A}=\holim{r}
  A_r$. Then Example \ref{E:ab4nhtyp} gives isomorphisms:
  \[
  \trunc{\geq p}\widehat{A} \simeq \trunc{\geq p}\left(\holim{r} A_r\right) \simeq \trunc{\geq p }\left(\holim{r}\trunc{\geq p-n-1}A_r\right) \simeq \trunc{\geq  p}\trunc{\geq p-n-1}A_{p-n-1} \simeq A_{p}.
  \]	
  The result follows.
\end{proof}
We now have the following theorem.
\begin{theorem}\label{T:gab-left-complete}
  Let $\mathcal{A}$ be a Grothendieck abelian category and let
  $\mathcal{M} \subseteq \mathcal{A}$ be a weak Serre subcategory that is
  closed under coproducts. Consider the functor
  \[
    \Psi \colon \DCAT(\mathcal{M}) \to \DCAT_{\mathcal{M}}(\mathcal{A}).
  \]
  Assume that
  \begin{enumerate}
  \item $\mathcal{M}$ is Grothendieck abelian;
  \item
    $\Psi^+ \colon \DCAT^+(\mathcal{M}) \to
    \DCAT^+_{\mathcal{M}}(\mathcal{A})$ is an equivalence; and
  \item $\DCAT_{\mathcal{M}}(\mathcal{A})$ is left faithful.
  \end{enumerate}
  Then the following assertions hold.
  \begin{enumerate}[label=(\alph*)]
  \item \label{TI:gab-left-complete:ff} If $\DCAT(\mathcal{M})$ is left faithful, then
    $\Psi$ is fully faithful.
  \item \label{TI:gab-left-complete:eq} If $\mathcal{M}$ satisfies
    \textup{AB}$4^*(\omega)n$ for some $n>0$, then $\Psi$ is an
    equivalence and $\DCAT_{\mathcal{M}}(\mathcal{A})$ is
    left complete with respect to the standard $t$-structure.
  \end{enumerate}
\end{theorem}
\begin{proof}
  In Example \ref{E:adjoints-prod-refined}, we have established that
  $\Psi$ admits a right adjoint $\Phi$ and that both $\Psi$ and $\Phi$
  restrict to an equivalence on $\DCAT^+$. We first establish
  \ref{TI:gab-left-complete:ff}. It suffices to prove that if
  $M\in \DCAT(\mathcal{M})$, then the naturally induced morphism
  $M \to \Phi \Psi(M)$ is an isomorphism. Now
  $\DCAT_{\mathcal{M}}(\mathcal{A})$ is left faithful, so
  $\Psi(M) \simeq \holim{r} \trunc{\geq -r}\Psi(M)$. But $\Phi$
  preserves homotopy limits (it is a right adjoint) and $\Psi$ is
  $t$-exact (it is the derived functor of the exact inclusion
  $\mathcal{M} \subseteq \mathcal{A}$); hence,
  \[
    \Phi\Psi(M) \simeq \holim{r} \Phi\Psi(\trunc{\geq -r}M) \simeq
    \holim{n} \trunc{\geq -r}M.
  \]
  But $\DCAT(\mathcal{M})$ is left faithful, so it follows that
  $M \to \Phi \Psi(M)$ is an isomorphism.

  We now prove \ref{TI:gab-left-complete:eq}. By Lemma
  \ref{L:left-complete}\itemref{LI:left-complete:complete},
  $\DCAT(\mathcal{M})$ is left complete and so left faithful. It
  follows from \ref{TI:gab-left-complete:ff} that $\Psi$ is fully
  faithful. Hence, it remains to prove that if
  $A \in \DCAT_{\mathcal{M}}(\mathcal{A})$ and $\Phi(A) =0$, then
  $A=0$. By assumption, $\DCAT_{\mathcal{M}}(\mathcal{A})$ is also
  left faithful, so we have an equivalence
  $A \simeq \holim{r} \trunc{\geq -r}A$ in
  $\DCAT_{\mathcal{M}}(\mathcal{A})$. Applying $\Phi$ to this
  equivalence, we obtain an equivalence
  $\Phi(A) \simeq \holim{r} \Phi(\trunc{\geq -r}A)$. Now $\Phi$ is
  $t$-exact on $\DCAT^+_{\mathcal{M}}(\mathcal{A})$. In particular, if
  $p\in \Z$, Example \ref{E:ab4nhtyp} gives
  \begin{align*}
    0&=\trunc{\geq p}\Phi(A) \simeq \trunc{\geq
       p}\holim{r}\Phi(\trunc{\geq -r}A) \simeq \trunc{\geq
       p}\holim{r}\trunc{\geq p-n-1}\Phi(\trunc{\geq -r}A)\\ 
     &\simeq \trunc{\geq
       p}\holim{r}\Phi(\trunc{\geq p-n-1}\trunc{\geq -r}A)\\
     &\simeq \trunc{\geq
       p}\holim{r}\Phi(\trunc{\geq p-n-1}A)\\
     &\simeq \trunc{\geq
       p}\Phi(\trunc{\geq p-n-1}A)\\
    &\simeq \Phi(\trunc{\geq p}A).
  \end{align*}
  It follows that $\trunc{\geq p}A\simeq 0$ for all $p\in \Z$. Since
  $\DCAT_{\mathcal{M}}(\mathcal{A})$ is left separated, $A = 0$.
\end{proof}
We immediately obtain the following corollary.
\begin{corollary}\label{C:equiv}
  Let $X$ be an algebraic stack. Assume
  \begin{enumerate}
  \item\label{CI:equiv:affdiag} $X$ is quasi-compact with affine diagonal; or
  \item\label{CI:equiv:affptd} $X$ is noetherian and affine-pointed.
  \end{enumerate}
  If $\QCOH(X)$ satisfies \textup{AB}$4^*n(\omega)$ for some $n$, then
  $\DCAT(\QCOH(X)) \to \DQCOH(X)$ is an equivalence.
\end{corollary}
Combining Corollary \ref{C:equiv} with Example \ref{E:ab4n-reg} gives
our second proof of Theorems \ref{MT:smooth} and \ref{T:reg-conc}.
\begin{remark}\label{R:hx}
  Hogadi--Xu \cite[Thm.~1.5]{2009arXiv0902.4016H} prove a result
  related to Theorem \ref{T:gab-left-complete}. The conclusion is
  similar and it has the additional assumption that $\mathcal{A}$ is
  also AB$4^*(\omega)n$. Our argument is a simple variant of
  theirs. They use this result to prove their Theorem 1.4 (our
  Corollary \ref{C:equiv}\itemref{CI:equiv:affdiag}). They do not
  appear to establish that $\MOD(X)$ satisfies AB$4^*n(\omega)$,
  however. While we were unable to find a counterexample, that
  $\MOD(X)$ satisfies AB$4^*n(\omega)$ in this generality appears to
  be non-obvious.
\end{remark}
\section{Direct proof of the Corollary in \S\ref{S:intro}}\label{S:direct-corollary}
Let $\pi \colon X \to Y$ be an affine and faithfully flat morphism of
algebraic stacks. We say that $\pi$ is a \emph{globally retracted
  cover} if $\Orb_Y \to \pi_*\Orb_X$ has a $\Orb_Y$-module
retraction. Derived variants of these were considered by Elagin
\cite{MR2830235}, where it was proved that globally retracted covers
allow one to do descent in the unbounded derived category. This
is also related to the work of Balmer on descent in triangulated
categories \cite{MR2910783}. The natural context to work in is that of
\emph{descendable} morphisms, which are discussed in the next
section. Since the Corollary in \S\ref{S:intro} is so simply stated and useful,
we thought it might be a good idea to provide a direct proof of it.

Globally retracted covers are not easy to find. We will locate some
interesting and useful ones, however. Our interest in them stems from the
following elementary result.
\begin{lemma}\label{L:splitting-cover-products}
  Let $\pi \colon X \to Y$ be a globally retracted cover. 
  Let $\Lambda$ be a set. If $\QCOH(X)$ satisfies
  \emph{AB}$4^*n(\Lambda)$ for some $n>0$, then so does $\QCOH(Y)$.
\end{lemma}
\begin{proof}
  Let $Q=\coker(\Orb_Y \to \pi_*\Orb_X)$. Then we have a split exact
  sequence:
  \[
    \xymatrix{0 \ar[r] & \Orb_Y \ar[r] & \pi_*\Orb_X \ar[r] & Q \ar[r] & 0.}
  \]
  Fix a retraction $r \colon \pi_*\Orb_X \to \Orb_Y$.  In particular,
  the sequence above remains exact after the application of \emph{any}
  additive functor to another abelian category. Hence, if
  $\{M_\lambda\}_{\lambda\in \Lambda}$ is a set of quasi-coherent
  sheaves on $Y$, then we obtain a set of split of exact sequences
  \[
    \xymatrix{0 \ar[r] & M_\lambda \ar[r] & \pi_*\Orb_X \tensor_{\Orb_Y} M_\lambda \ar[r] & Q \tensor_{\Orb_Y} M_\lambda \ar[r] & 0,}
  \]
  each of which is split by
  $r_\lambda \colon \pi_*\Orb_X \tensor_{\Orb_Y} M_\lambda \to
  M_\lambda$. This family of exact sequences corresponds to a split
  exact sequence in the product category $\QCOH(Y)^\Lambda$. Since
  $\pi$ is affine and flat, we also have a projection isomorphism
  $\pi_*\Orb_X \tensor_{\Orb_Y} M_\lambda \simeq \pi_*\pi^*M_\lambda$
  in $\QCOH(Y)$. It follows that for all $s>0$ that we have an
  injection in $\QCOH(Y)$:
  \[
    \prod^{(s)}_{\lambda\in \Lambda} M_\lambda \hookrightarrow
    \prod^{(s)}_{\lambda\in \Lambda} (\pi_*\Orb_X \tensor_{\Orb_Y}
    M_\lambda) \simeq \prod^{(s)}_{\lambda\in \Lambda}
    \pi_*\pi^*M_\lambda \simeq \pi_*\left(\prod^{(s)}_{\lambda\in
        \Lambda} \pi^*M_\lambda\right).
  \]
  The final equivalence follows from Example \ref{E:adjoints-prod},
  applied to the adjoint pair $(\pi^*,\pi_*)$. The result follows.
\end{proof}
We now go about finding useful examples.
\begin{example}\label{E:gln-torus-split}
  If $k$ is a field of characteristic $0$, then
  $\spec k \to B\GL_{s,k}$ is globally split $\forall s>0$.
\end{example}
The above example can be generalized to
linearly reductive group schemes.
\begin{example}\label{E:globally-split-functorial}
  Let $f \colon Y' \to Y$ be a morphism of algebraic stacks. If
  $\pi\colon X \to Y$ is a globally split cover, then
  $\pi' \colon X'=X\times_Y Y' \to Y'$ is a globally split
  cover. 
\end{example}

\begin{proof}[Proof of the Corollary in \S\ref{S:intro}]
  Since $G$ is an affine group scheme over $k$, it admits an embedding
  $G \hookrightarrow \GL_{s,k}$ for some $s>0$. Then
  $[U/G] \simeq [V/\GL_{s,k}]$, where
  $V \simeq U \times (\GL_{s,k}/G)$ with the diagonal action of
  $\GL_{s,k}$. Then $V$ is noetherian and Deligne--Mumford. It follows
  by combining Examples \ref{E:gln-torus-split} and
  \ref{E:globally-split-functorial} that
  $V \to [V/\GL_{s,k}]\simeq [U/G]$ is globally retracted. By
  Proposition \ref{P:qfdiag_ab4starn}, $V$ satisfies AB$4^*n$ for some
  $n>0$. By Lemma \ref{L:splitting-cover-products}, $[U/G]$ satisfies
  AB$4^*n$. By Corollary \ref{C:equiv} we have the result.
\end{proof}
\section{Descendable morphisms}\label{S:desc_bg}
Here we review the elegant discussion of \emph{descendable} morphisms,
due to \cite{MR3459022}, which appears in \cite[\S2 \&
\S4]{aoki2020quasiexcellence}. This is closely related to the work of
Balmer for separable algebras \cite{2013arXiv1309.1808B}.

Let $\mathcal{S}$ be a triangulated category. A collection of objects
of $\mathcal{S}$ is \emph{closed} if it is closed under finite
coproducts and direct summands. If $S$ is a collection of objects of
$\mathcal{S}$, let $\closure{S}$ denote its \emph{closure}
(i.e., the smallest closed collection of objects of $\mathcal{S}$ that
contains it).  If $S_1$, $S_2$ are closed collections of objects of
$\mathcal{S}$, define $S_1 \star S_2$ to be the closure of the
collection
\[
  \{ s \in \mathcal{S} \suchthat \mbox{$s$ sits in a distinguished
    triangle $s_1 \to s \to s_2 \to s_1[1]$, where $s_i \in S_i$}\}.
\]
Let $\mathcal{B}$ be an
abelian category and let $\shfcoho \colon \mathcal{S} \to \mathcal{B}$
be a homological functor. If $s\in \mathcal{T}$, define its
$\shfcoho$-amplitude, $\amp_{\shfcoho}(s)$, to be the smallest
connected interval of $\Z$ with $\shfcoho^i(s) = 0$ if
$i\notin \amp_{\shfcoho}(t)$.  

We have the following trivial lemma.
\begin{lemma}\label{L:amplitude-star}
  Let $\shfcoho \colon \mathcal{S} \to \mathcal{B}$ be a homological functor. Let $S_1$, $\dots$, $S_k$ be a collection of objects of
  $\mathcal{S}$. If $s\in \closure{S_1} \star \cdots \star  \closure{S_k}$, then
  \[
    \amp_{\shfcoho}(s) \subseteq \bigcup_{i=1}^k\bigcup_{y\in S_i} \amp_{\shfcoho}(y).
  \]
\end{lemma}
\begin{example}
  Let $\mathcal{A}$ be a Grothendieck abelian category. We have the homological functor
  $\COHO{} \colon \DCAT(\mathcal{A}) \to \mathcal{A}$ that sends $M$ to its $0$th
  cohomology group. The $\COHO{}$-amplitude, we just refer to as
  \emph{amplitude}.
\end{example}
Now suppose that $(\mathcal{S},\otimes)$ is a tensor
triangulated category (i.e., $\mathcal{S}$ is a triangulated category,
is symmetric monoidal, and the $\tensor$ is compatible with triangles
and shifts). Let $A$ be an $\mathcal{S}$-algebra. 
Then $A$ is \emph{descendable} if the smallest thick tensor ideal of
$\mathcal{S}$ containing $A$ is $\mathcal{S}$.

Form the distinguished triangle:
\[
\xymatrix{  K_{A} \ar[r]^{\phi_A} & 1_{\mathcal{S}} \ar[r]^{\eta_A} & A \ar[r]^{\delta_A} & K_A[1]}
\]
We say that $A$ is \emph{descendable of index $\leq d$} if
$\phi_A^{\otimes d} \colon K_A^{\otimes d} \to 1_{\mathcal{S}}$ is
$0$. For each non-negative integer $i$, form a distinguished triangle:
\begin{equation}
  \xymatrix{K_A^{\otimes i} \ar[r]^{\phi_A^{\otimes i}} &
    1_{\mathcal{S}} \ar[r]^{\eta_A^i} & Q_A^i \ar[r]^{\delta_A^i} &
    K_A^{\otimes i}[1].}\label{eq:Q-define}
\end{equation}
It follows that if $A$ is descendable of index $\leq d$, then
$1_{\mathcal{S}}$ is a summand of $Q^{d}_A$. We now have a morphism of triangles:
\[
  \xymatrix@C+1.5pc{K_A^{\otimes (i+1)} \ar[d]_-{\mathrm{id}_{K_A^{\otimes i}} \otimes \phi_A}\ar[r]^{\phi_A^{\otimes (i+1)}} &
    1_{\mathcal{S}} \ar@{=}[d]\ar[r]^{\eta_A^{i+1}} & Q_A^{i+1} \ar@{-->}[d]^{\exists \chi_A^{i}}\ar[r]^{\delta_A^{i+1}} &
    K_A^{\otimes (i+1)}[1] \ar[d]^-{\mathrm{id}_{K_A^{\otimes i}} \otimes \phi_A[1]}\\
    K_A^{\otimes i} \ar[r]^{\phi_A^{\otimes i}} & 1_{\mathcal{S}}
    \ar[r]^{\eta_A^i} & Q_A^i \ar[r]^{\delta_A^i} & K_A^{\otimes
      i}[1].}
\]
From the octahedral axiom \cite[Prop.~1.4.6]{MR1812507}, we now obtain the
following commutative diagram with distinguished rows and columns:
\[
  \xymatrix@C+1.5pc{ 
    0 \ar[r] \ar[d] & K^{\otimes i}_A \otimes A[-1]
    \ar[d] \ar@{=}[r]& K_A^{\otimes i}\ar[d] \otimes A[-1] \ar[r] & 0 \ar[d]\\
    1_{\mathcal{S}}[-1] \ar[r]^-{\eta_A^{i+1}[-1]} \ar@{=}[d] & Q_A^{i+1}[-1] \ar[r]^-{\delta_A^{i+1}[-1]}
    \ar[d]^-{\chi_A^i[-1]} & \ar[d] K_A^{\otimes (i+1)} \ar[r]^-{\phi_A^{\otimes
        (i+1)}} & 1_{\mathcal{S}} \ar@{=}[d]
    \\ 1_{\mathcal{S}}[-1] \ar[r]^-{\eta_A^i[-1]} \ar[d] &Q_A^i[-1] \ar[d] \ar[r]^-{\delta_A^{i}[-1]} & K_A^{\otimes i} \ar[r]^-{\phi_A^{\otimes i}}\ar[d] &
    1_{\mathcal{S}} \ar[d] \\ 0 \ar[r] & K^{\otimes i}_A \otimes A
    \ar@{=}[r]& K_A^{\otimes i} \otimes A \ar[r] & 0}
\]
This gives the distinguished triangle:
\begin{equation}
  \xymatrix{K^{\otimes i}_A \otimes A \ar[r] & Q^{i+1}_A \ar[r] & Q_A^i \ar[r] & K_A^{\otimes i} \otimes A[1].}\label{eq:Q-cones}
\end{equation}
In particular, if $A$ is descendable of index $\leq d$ and $s \in \mathcal{S}$, then 
\begin{equation}
  s \in \closure{(s \otimes K_A^{\otimes d-1} \otimes A)}
  \star \closure{(s \otimes K_A^{\otimes d-2} \otimes A)} \star \cdots \star \closure{(s\otimes A)}.\label{eq:desc-index}
\end{equation}
Hence, if $A$ is descendable of index $\leq d$, then $A$ is
descendable.  The converse also holds
(\cite[Prop.~3.15]{2013arXiv1309.1808B}, \cite[Prop.~3.26]{MR3459022},
and \cite[Lem.~11.20]{MR3674218}), but we will not need it for our
main results.
\begin{example}\label{E:lazard}
  Let $R$ be a ring. Let $M$ be a countably presented $R$-module. If
  $M$ is flat over $R$, then
  \[
    \Ext^r_R(M,N) = 0
  \]
  for all $r>1$ and $R$-modules $N$
  \cite[Thm.~3.2]{lazard_autour}. Now let $A$ be a faithfully flat and
  countably presented $R$-algebra. Then $R$ is a flat and countably
  presented $R$-module. Let $L=\coker (R \to A)$; then $L$ is also a
  flat and countably presented $R$-module ($L \otimes_R A$ is a direct
  summand of $A\otimes_R A$). In particular, $L^{\otimes 2}$ is a flat
  and countably presented $R$-module.  But $K_A=L[-1]$, so
  \[
    \Hom_R(K_A^{\otimes 2},R) \simeq \Ext^2_R(L^{\otimes 2},R) \simeq 0.
  \]
  Hence, $A$ is descendable of index $\leq 2$.
\end{example}
\begin{example}\label{E:descendable_counter}
  Let $X=B_{\F_2}(\Z/2\Z)$. Consider $\mathcal{S}=\DQCOH(X)$. Let
  $\pi \colon \spec \F_2 \to X$ be the standard finite \'etale
  covering. Then $A=\RDERF \pi_*\Orb_{\spec \F_2}$ (the ring of
  regular functions on $\Z/2\Z$) is not descendable in
  $\DQCOH(X)$. Indeed, it is easily calculated that
  $K_A\simeq \Orb_X[-1]$. It follows from Equation
  \eqref{eq:desc-index} that if $A$ is descendable of index $\leq d$,
  then $\Orb_X$ belongs to
  $\closure{A}^{\star d} \subseteq \DQCOH(X)^c$. But this is
  impossible: $\Orb_X \notin \DQCOH(X)^c$
  \cite[Rem.~4.6]{perfect_complexes_stacks}.
\end{example}
We now have a key lemma.
\begin{lemma}\label{L:amplitude-tor-star}
  Let $(\mathcal{S},\otimes)$ be a tensor triangulated category. Let
  $\shfcoho \colon \mathcal{S} \to \mathcal{B}$ be a homological
  functor. Let $A$ be a descendable $\mathcal{S}$-algebra of index $\leq d$. If
  $s\in \mathcal{S}$, then
  \[
    \amp_{\shfcoho}(s) \subseteq \bigcup_{i=1}^{d}\amp_{\shfcoho}(s\otimes K_A^{\otimes i-1} \otimes A).
  \]
\end{lemma}
\begin{proof}
  Combine Lemma \ref{L:amplitude-star} with Equation
  \eqref{eq:desc-index}.
\end{proof}
If $X$ is an algebraic stack of finite cohomological
dimension, let $\cd(X)$ be its cohomological dimension. Let
$\pi\colon X \to Y$ be a \emph{concentrated} morphism of quasi-compact
and quasi-separated algebraic stacks; that is, $X\times_Y \spec A$ has
finite cohomological dimension for all affine schemes $\spec A$ and
morphisms $\spec A \to Y$ \cite[\S2]{perfect_complexes_stacks}. Let
\[
  \cd(\pi) = \max\{\cd(X\times_Y V) \suchthat
  \mbox{$V$ is an affine object of $Y_{\lisset}$}\}
\]
It follows from \cite[Lem.~2.2]{perfect_complexes_stacks} that
$\cd(\pi) <\infty$.
\begin{lemma}\label{L:amplitude-stacks}
  Let $\pi \colon U \to X$ be a quasi-compact and quasi-separated
  morphism of algebraic stacks. Let
  $\shfcoho \colon \DQCOH(X) \to \mathcal{B}$ be a cohomological
  functor. Assume that $\RDERF \pi_*\Orb_U$ is descendable of index
  $\leq d$. If $\pi$ is concentrated and $M \in \DQCOH(X)$, then
  \[
    \amp_{\shfcoho}(M) \subseteq \bigcup_{i=1}^d \amp_{\shfcoho}(\RDERF
    \pi_*\LDERF \pi^*(M\otimes^{\LDERF}_{\Orb_X} K_{\RDERF
      \pi_*\Orb_U}^{\otimes i-1})).
  \]
\end{lemma}
\begin{proof}
  Combine Lemma \ref{L:amplitude-tor-star} with the projection formula
  \cite[Cor.~4.12]{perfect_complexes_stacks}, which shows that
  $\RDERF \pi_*\LDERF \pi^*(M\otimes^{\LDERF}_{\Orb_X} K_{\RDERF
    \pi_*\Orb_U}^{\otimes i-1}) \simeq (M\otimes^{\LDERF}_{\Orb_X}
  K_{\RDERF \pi_*\Orb_U}^{\otimes i-1}) \otimes^{\LDERF}_{\Orb_X}
  \RDERF \pi_*\Orb_U$.
\end{proof}
The following example improves upon Example
\ref{E:descendable_counter}.
\begin{example}
  Let $\pi \colon U \to X$ be a concentrated morphism of algebraic
  stacks that is descendable of index $\leq d$. If $U$ has finite
  cohomological dimension, then so does $X$. Indeed, if $W$ is an
  algebraic stack, consider the homological functor
  $\shfcoho_W \colon \DQCOH(W) \to \AB$ that sends $M$ to
  $\shfcoho^0(W,M)$. It follows immediately that
  $\shfcoho_X(\RDERF \pi_*(-)) = \shfcoho_U(-)$.  Lemma
  \ref{L:amplitude-stacks} implies that if $M\in \QCOH(X)$, then
  \[
    \amp_{\shfcoho_X}(M) \subseteq \bigcup_{i=1}^d
    \amp_{\shfcoho_U}(\LDERF\pi^*(M \otimes^{\LDERF}_{\Orb_X}
    K^{\otimes i-1}_{\RDERF \pi_*\Orb_U})) \subseteq (-\infty,(d-1)(\cd(\pi)+1)].
  \]
  Hence, $\cd(X)\leq (d-1)(\cd(\pi)+1)$.
\end{example}

Our interest in descendable morphisms is because of the following.
\begin{proposition}\label{P:desc-prods}
  Let $\pi \colon U \to X$ be a concentrated, faithfully flat, and
  finitely presented morphism of algebraic stacks. Let $\Lambda$ be a set. Assume that
  \begin{enumerate}
  \item $\RDERF \pi_*\Orb_U$ is
    descendable in $\DQCOH(X)$;
  \item $\QCOH(U)$ satisfies \textup{AB}$4^*n(\Lambda)$ for some $n$; and 
  \item $\DCAT^+(\QCOH(U)) \simeq \DQCOH^+(U)$ and
    $\DCAT^+(\QCOH(X)) \simeq \DQCOH^+(X)$.
  \end{enumerate}
  Then $\QCOH(X)$ satisfies \textup{AB}$4^*m(\Lambda)$ for some $m$.
\end{proposition}
Unfortunately, we are not yet in a position to prove Proposition
\ref{P:desc-prods}. We will first need to
establish some cohomological estimates for $\RDERF \pi_*\Orb_U$.
\section{Descendable morphisms of stacks}\label{S:desc-non-aff}
In this section, we establish that a wide class of morphisms of
algebraic stacks are descendable.
\begin{theorem}\label{T:desc-non-aff}
  Let $\pi \colon U \to X$ be a concentrated, faithfully flat, and
  finitely presented morphism of algebraic stacks. If $X$ has finite
  cohomological dimension, then $\RDERF\pi_*\Orb_U$ is descendable of index $\cd(X)+3$ in
  $\DQCOH(X)$ when
  \begin{enumerate}
  \item \label{TI:desc-non-aff:tame}$\pi$ is tame (e.g., representable); or
  \item \label{TI:desc-non-aff:aff}$\pi$ has affine stabilizers and
    $X$ has equicharacteristic; or
  \item \label{TI:desc-non-aff:countable} $X$ is noetherian.
  \end{enumerate}
\end{theorem}
The conditions
\itemref{TI:desc-non-aff:tame}-\itemref{TI:desc-non-aff:countable}
arise in Theorem \ref{T:desc-non-aff} because of the poor
approximation properties of concentrated morphisms with infinite
stabilizers in mixed characteristic
\cite[\S2]{hallj_dary_alg_groups_classifying}. Example
\ref{E:descendable_counter} shows that $X$ having finite cohomological
dimension is essential in Theorem \ref{T:desc-non-aff}.

We will prove Theorem \ref{T:desc-non-aff} by establishing global
vanishing results for certain complexes of sheaves, which we now
define.

Now let $N$ be a quasi-coherent $\Orb_X$-module. Consider the
homological functor
$\Hom_{\Orb_X}(-,N) \colon \DQCOH(X) \to \AB^{\opp}$. By the
discussion in \S\ref{S:desc_bg}, we obtain a notion of
$\Hom(-,N)$-amplitude. We define the \emph{$\Hom$-amplitude} of $M$ to
be
\[
  \amp_{\Hom}(M) = \bigcup_{N} \amp_{\Hom(-,N)}(N).
\]
\begin{example}
  Because of the contravariance, the $\Hom$-amplitude can be
  confusing. We offer the following simple example. Let
  $X=\mathbf{P}^1_k$, where $k$ is a field, and let
  $\pi \colon X \to \spec k$ be the structure map. Then
  $\RDERF \pi_*\Orb(-2)$ has $\Hom$-amplitude $[-1,-1]$ on $\spec
  k$. Indeed, a standard calculation shows that
  $\RDERF \pi_*\Orb(-2) \simeq \Orb_{\spec k}[-1]$. While this has
  cohomological amplitude $[1,1]$ in $\DCAT(k)$, the $\Hom$-amplitude
  is defined in the opposite category, so there is a reversing of
  degrees (i.e., homological grading).
\end{example}
We also define the \emph{$\SHom$-amplitude} as:
\[
  \amp_{\SHom}(M) = \bigcup_{V} \amp_{\Hom} (M_V),
\]
where $V$ ranges over all affine objects of $X_{\lisset}$. The
following result follows easily from arguments similar to those in
\cite[Tag \spref{0B66}]{stacks-project}.
\begin{lemma}\label{L:hom-amp-properties}
  Let $X$ be an algebraic stack. Let $P\in \DQCOH(X)$.
  \begin{enumerate}
  \item\label{LI:hom-amp-properties:cplx} If $P$ has $\SHom$-amplitude
    $[a,b]$, then $P$ is locally quasi-isomorphic to a complex of
    projective $\Orb_X$-modules supported only in cohomological
    degrees $[-b,-a]$.
  \item\label{LI:hom-amp-properties:tensor} If $P$ has
    $\SHom$-amplitude $[a,b]$ and $P'$ has $\SHom$-amplitude
    $[a',b']$, then $P\otimes^{\LDERF}_{\Orb_X} P'$ has
    $\SHom$-amplitude contained in $[a+a',b+b']$.
  \item\label{LI:hom-amp-properties:flat} If $\pi \colon U \to X$ is
    faithfully flat and $\LDERF \pi^*P$ has $\SHom$-amplitude $[a,b]$,
    then so does $P$.
  \item\label{LI:hom-amp-properties:summand} If $Q$ is a direct
    summand of $P$, and $P$ has $\SHom$-amplitude $[a,b]$, then $Q$
    has $\SHom$-amplitude contained in $[a,b]$.
  \end{enumerate}
\end{lemma}
We now have our amplitude results.
\begin{theorem}\label{T:ext-vanishing}
  Let $X$ be a quasi-compact and quasi-separated algebraic stack. If
  $X$ has finite cohomological dimension, then there exists 
  a non-negative integer $d_X$ with the following property: if
  $Q\in \DQCOH(X)$ has $\SHom$-amplitude $[a,b]$, then it has
  $\Hom$-amplitude contained in $[a,b+d_X]$. 
\end{theorem}
\begin{theorem}\label{T:push}
  Let $\pi \colon U \to X$ be a flat and finitely presented morphism
  of algebraic stacks. If $\pi$ is concentrated, then there is a
  non-negative integer $e_\pi$ such that if $P$ has $\SHom$-amplitude
  $[a,b]$ on $U$, then $\RDERF \pi_*P$ has $\SHom$-amplitude contained
  in $[a-e_\pi,b+1]$.
\end{theorem}
\begin{corollary}\label{C:K-bound}
  Let $\pi \colon U \to X$ be a faithfully flat and finitely presented
  morphism of algebraic stacks. If $\pi$ is concentrated, then for all
  non-negative integers $i$, $K_{\RDERF \pi_*\Orb_U}^{\otimes i}$ has
  $\SHom$-amplitude contained in $[-i(e_\pi+1),-i+1]$.
\end{corollary}
\begin{remark}
  If $\pi$ is affine, then we can take $e_\pi=0$. If $\pi$ has affine
  diagonal, then we can take $e_\pi = \cd(\pi)+2$. Also, if $X$ has
  affine diagonal, then $d_X=\cd(X)+2$.
\end{remark}
With these results in hand, we can now prove Proposition \ref{P:desc-prods}.
\begin{proof}[Proof of Proposition \ref{P:desc-prods}]
  Let $\Lambda$ be a set. Let $\{N_\lambda\}_{\lambda\in \Lambda}$ be
  a set of quasi-coherent $\Orb_X$-modules. Let
  $A=\RDERF \pi_*\Orb_U$, which is descendable of index $\leq d$. Consider the homological
  functor $F \colon \DQCOH(X) \to \QCOH(X)$ that sends
  $M \in \DQCOH(X)$ to
  $\COHO{0}(\prod_{\lambda\in \Lambda}^{\DQCOH(X)} (N_\lambda
  \otimes^{\LDERF}_{\Orb_X} M))$. By Lemma \ref{L:amplitude-stacks},
   \[
     \amp_F(\Orb_X) \subseteq \bigcup_{i=1}^d \amp_F(\RDERF \pi_*\LDERF \pi^*(K_A^{\otimes i-1})).
  \]
  But the projection formula
  \cite[Cor.~4.12]{perfect_complexes_stacks} implies that
  \begin{align*}
\prod_{\lambda\in \Lambda}^{\DQCOH(X)} (N_\lambda 
    \otimes^{\LDERF}_{\Orb_X} \RDERF \pi_*\LDERF \pi^*(K_A^{\otimes i-1})) &\simeq \prod_{\lambda \in \Lambda}^{\DQCOH(X)} \RDERF \pi_*(\pi^*N_\lambda \otimes^{\LDERF}_{\Orb_U} \LDERF \pi^*K_A^{\otimes i-1})\\
    &\simeq \RDERF \pi_* \left( \prod^{\DQCOH(U)}_{\lambda\in \Lambda} \pi^*N_\lambda \otimes^{\LDERF}_{\Orb_U} \LDERF \pi^*K_A^{\otimes i-1}\right).
  \end{align*}
  By Corollary \ref{C:K-bound}, $K_A^{\otimes i-1}$ has
  $\SHom$-amplitude contained in $[-(e_\pi+1)(i-1),0]$. It follows from
  Lemma \ref{L:hom-amp-properties}(\ref{LI:hom-amp-properties:cplx},\ref{LI:hom-amp-properties:tensor}) that the amplitude of
  $\pi^*N_\lambda \otimes^{\LDERF}_{\Orb_U} \LDERF \pi^*K_A^{\otimes
    i-1}$ is contained in $[0,(e_\pi+1)(i-1)]$. By Example
  \ref{E:adjoints-prod-refined} applied to
  $\QCOH(U) \subseteq \MOD(U)$, the amplitude of
  $\prod_{\lambda\in \Lambda}^{\DQCOH(U)} \pi^*N_\lambda
  \otimes^{\LDERF}_{\Orb_U} \LDERF \pi^*K_A^{\otimes i-1}$ is
  contained in $[0,(e_\pi+1)(i-1)+n]$. It follows that
  $\RDERF \pi_*\left(\prod_{\lambda\in \Lambda}^{\DQCOH(U)}
    \pi^*N_\lambda \otimes^{\LDERF}_{\Orb_U} \LDERF \pi^*K_A^{\otimes
      i-1} \right)$ has amplitude contained in
  $[0,(e_\pi+1)(i-1)+n+\cd(\pi)]$. Putting everything together,
  \[
    \amp_F(\Orb_X) \subseteq [0,(e_\pi+1)(d-1)+n+\cd(\pi)].
  \]
  Now take $m=(e_\pi+1)(d-1)+n+\cd(\pi)$. Then
  $\trunc{>m}\prod_{\lambda\in \Lambda}^{\DQCOH(X)} N_\lambda \simeq
  0$. Since we have the equivalence
  $\DCAT^+(\QCOH(X)) \simeq \DQCOH^+(X)$, the result follows from
  Example \ref{E:adjoints-prod-refined} applied to
  $\QCOH(X) \subseteq \MOD(X)$.
\end{proof}
\begin{proof}[Proof of Theorems \ref{T:desc-non-aff}, \ref{T:ext-vanishing}, \ref{T:push} and Corollary \ref{C:K-bound}]
  \renewcommand\qedsymbol{$\blacksquare$} We prove these results
  simultaneously via a bootstrapping process. We have the following
  assertions that we will establish. Step \itemref{proof:loc-glob-2}
  is the most difficult. If $\gamma \colon T\to S$ is a morphism, for
  each $r\geq 1$, let $(T/S)^r$ be the $r$th fiber product of $T$ over
  $S$. Let $\gamma^r \colon (T/S)^r \to S$ be the induced morphism. Let $A=\RDERF \pi_*\Orb_U$.
  \begin{enumerate}
  \item \label{proof:cor} Theorem \ref{T:push} for $\pi^r$ and
    $P=\Orb_{(U/X)^r}$,  $r\leq i$ $\implies$ Corollary
    \ref{C:K-bound} for $\pi$ and $i$.
  \item \label{proof:hom} Let $p\colon W \to U$ be a smooth surjection
    from an affine scheme $W$. Theorem \ref{T:desc-non-aff} for
    $p$ + Theorem \ref{T:push} for $p$, $\pi\circ p$ $\implies$
    Theorem \ref{T:push} for $\pi$.
  \item \label{proof:desc} Let $\rho \colon V \to X$ be a faithfully
    flat, finitely presented, and concentrated morphism. Theorems
    \ref{T:desc-non-aff} + \ref{T:push} for $\rho$ + Theorem
    \ref{T:ext-vanishing} for $V$ $\implies$ Theorem
    \ref{T:ext-vanishing} for $X$.
  \item \label{proof:loc-glob-2} Theorem \ref{T:push} for $\pi^r$, $r \leq \cd(X)+3$ 
    $\implies$ Theorem \ref{T:desc-non-aff} for $\pi$.    
  \end{enumerate}
  
  Let us put these together. We first prove Theorem \ref{T:push} when
  $\pi$ is affine. We may immediately reduce to the case where
  $X=\spec R$ and $U=\spec A$ are affine. Then $P$ is quasi-isomorphic
  to a complex of projective $R$-modules supported in cohomological
  degrees $[-b,-a]$ (Lemma
  \ref{L:hom-amp-properties}\itemref{LI:hom-amp-properties:cplx}). A
  simple argument using brutal truncations shows that it suffices to
  prove the result when $P$ is supported in degree $0$. Then $P$ is a
  direct summand of a free $A$-module. But $\RHom(-,N)$ sends $\oplus$
  to $\prod$, so it suffices to prove the result when $P=A$. We are
  now reduced to Example \ref{E:lazard}, and we can deduce the result.

  By \itemref{proof:cor}, we have Corollary \ref{C:K-bound} for all
  affine $\pi$ and $i\geq 1$.  It now follows from
  \itemref{proof:loc-glob-2} that Theorem \ref{T:desc-non-aff} holds
  when $\pi$ is affine.

  We next prove that Theorem \ref{T:push} holds when $\pi$ has affine
  diagonal. As before, we may reduce to the case where $X=\spec R$ and
  $U$ is a quasi-compact algebraic stack with affine diagonal. Let
  $p \colon W \to U$ be a smooth surjection from an affine scheme
  $W$. Then $p$ and $\pi\circ p$ are affine. It follows from
  \itemref{proof:hom} that Theorem \ref{T:push} holds for $\pi$. It
  now follows from \itemref{proof:cor} and \itemref{proof:loc-glob-2}
  that Theorem \ref{T:desc-non-aff} holds when $\pi$ has affine
  diagonal.

  A similar argument proves that Theorems \ref{T:push} and so Theorem
  \ref{T:desc-non-aff} holds when $\pi$ has affine second
  diagonal. Repeating, we get Theorems \ref{T:push} and
  \ref{T:desc-non-aff} when $\pi$ has affine third diagonal. But for
  every algebraic stack, the third diagonal is an isomorphism, so we
  have it for all $\pi$.
  
  Let $\rho \colon V \to X$ be a smooth surjection from an affine
  scheme $V$. Then Theorem \ref{T:desc-non-aff} holds for $\rho$. By
  \itemref{proof:hom}, Theorem \ref{T:ext-vanishing} holds for $X$.

  We now go about proving the assertions \itemref{proof:cor}-\itemref{proof:loc-glob-2}.
  \begin{proof}[Proof of \itemref{proof:cor}]
    Clearly,
    $K_A \in \closure{A[-1]} \star \closure{\Orb_X}$. By
    Theorem \ref{T:push} we have that $A[-1]$ has $\SHom$-amplitude
    $[-e_\pi-1,0]$. Trivially, $\Orb_X$ has $\SHom$-amplitude
    $[0,0]$. It follows from Lemma \ref{L:amplitude-star} that $K_A$
    has $\SHom$-amplitude contained in $[-e_\pi-1,0]$. Lemma
    \ref{L:hom-amp-properties} gives the trivial estimate of
    $[-i(e_\pi+1),0]$ for the $\SHom$-amplitude, which will be useful,
    but it turns out we can do much better. Tensoring the defining
    triangle for $K_A$ with
    $K_A^{\otimes i-1} \otimes^{\LDERF}_{\Orb_X} A^{\otimes j}$ we
    produce a triangle:
    \[
      \xymatrix@C-1.2pc{K^{\otimes i-1}_A \otimes_{\Orb_X}^{\LDERF} A^{\otimes
          j+1} \ar[r] & K_A^{\otimes i} \otimes_{\Orb_X}^{\LDERF}
        A^{\otimes j} \ar[r] & K_A^{\otimes i-1} \otimes_{\Orb_X}^{\LDERF}
        A^{\otimes j}\ar[r] & K^{\otimes
          i-1}_A \otimes_{\Orb_X}^{\LDERF} A^{\otimes j+1}[1].}
    \]
    The pullback of this sequence along $\pi$ is split and $\pi$ is
    faithfully flat. Hence,
    $K^{\otimes i}_A \otimes_{\Orb_X}^{\LDERF} A^{\otimes j}$ has
    $\SHom$-amplitude contained in the $\SHom$-amplitude of
    $K^{\otimes i-1}_A \otimes^{\LDERF}_{\Orb_X}A^{\otimes
      j+1}[-1]$. Starting with $(i,0)$ and repeating, we find that
    $K^{\otimes i}_A$ has $\SHom$-amplitude contained in
    $A^{\otimes i}[-i]$. But
    $A^{\otimes i} \simeq \RDERF \pi^i_*{\Orb_{(U/X)^i}}$, which has
    $\SHom$-amplitude contained in $[-e_{\pi^i}, 1]$. Hence,
    $K^{\otimes i}_A$ has $\SHom$-amplitude contained in
    $[-(e_{\pi^i}+i),-i+1]$. From the trivial estimate,
    $K_A^{\otimes i}$ has $\SHom$-amplitude $[-i(e_\pi+1),0]$.
    Hence, $K_A^{\otimes i}$ has
    $\SHom$-amplitude contained in $[-i(e_\pi+1),-i+1]$.
  \end{proof}
  \begin{proof}[Proof of \itemref{proof:hom}]
    Let $V$ be an affine object of $X_{\lisset}$ and $N$ a
    quasi-coherent $\Orb_V$-module. Let $\gamma \colon Y \to X$ be a
    concentrated morphism of algebraic stacks. Let
    $\mathcal{H}_{\gamma,V,N} \colon \DQCOH(Y) \to \AB^{\opp}$ be the
    homological functor
    $\shfcoho^0(V,\SRHom_{\Orb_V}((\RDERF \gamma_*(-))_V,N))$. It
    suffices to prove that there is an $e_\pi$ such that if $P$ has
    $\SHom$-amplitude $[a,b]$, then
    $\amp_{\mathcal{H}_{\pi,V,N}}(P) \subseteq [a-e_\pi,b+1]$ for all
    $V$ and $N$. We also have an isomorphism of homological functors:
    \[
      \mathcal{H}_{\pi,V,N}(\RDERF p_*(-)) \simeq \mathcal{H}_{\pi\circ p,V,N}(-).
    \]
    Now $\RDERF
    p_*\Orb_W$ is descendable of index $\leq d=\cd(U)+3$.  Let $P\in
    \DQCOH(U)$ have $\SHom$-amplitude
    $[a,b]$; then Lemma \ref{L:amplitude-stacks} implies that
    \begin{align*}
      \amp_{\mathcal{H}_{\pi,V,N}}(P)
      &\subseteq \bigcup_{i=1}^d
        \amp_{\mathcal{H}_{\pi,V,N}}(\RDERF p_*\LDERF p^*(P
        \otimes^{\LDERF}_{\Orb_W} K_{\RDERF p_*\Orb_W}^{\otimes i-1})) \\
      &= \bigcup_{i=1}^d      \amp_{\mathcal{H}_{\pi\circ p,V,N}}(\LDERF p^*(P
        \otimes^{\LDERF}_{\Orb_W} K_{\RDERF p_*\Orb_W}^{\otimes i-1}))
    \end{align*}
    By \itemref{proof:cor}, we have that
    $K_{\RDERF p_*\Orb_W}^{\otimes i-1}$ has $\SHom$-amplitude
    $[-(e_p+1)(i-1),0]$. It follows from Lemma
    \ref{L:hom-amp-properties} that
    $\LDERF p^*(P \otimes^{\LDERF}_{\Orb_W} K_{\RDERF
      p_*\Orb_W}^{\otimes i-1})$ has $\SHom$-amplitude contained in
    $[-(e_p+1)(i-1)+a,b]$. Hence,
    \[
      \amp_{\mathcal{H}_{\pi,V,N}}(P) \subseteq [-(e_p+1)(i-1)+a-e_{\pi \circ p},b+1].
    \]
    Taking $e_\pi = (e_p+1)(\cd(U)+2)+e_{\pi \circ p}$ gives the claim.     
  \end{proof}
  \begin{proof}[Proof of \itemref{proof:desc}]
    Let $P \in \DQCOH(X)$ have $\SHom$-amplitude $[a,b]$. Let
    $H_P \colon \DQCOH(X) \to \AB$ be the homological functor
    $\Hom_{\Orb_X}(P,-)$. It suffices to find $d_X$, independent of
    $P$, such that if $N$ is a quasi-coherent $\Orb_X$-module, then
    $\amp_{H_P}(N) \subseteq [a,b+d_X]$. Let $d_\rho=\cd(X)+3$, which
    is the descent index of $\rho$. Then Lemma
    \ref{L:amplitude-stacks} implies that
    \[
      \amp_{H_P}(N) \subseteq \bigcup_{i=1}^{d_\rho} \amp_{H_P}(\RDERF
      \rho_*(\rho^*N \otimes^{\LDERF}_{\Orb_V} \LDERF \rho^*K_{\RDERF
        \rho_*\Orb_V}^{\otimes i-1})).
    \]
    By adjunction:
    \[
      \RHom_{\Orb_X}(P,\RDERF \rho_*(\rho^*N \otimes^{\LDERF}_{\Orb_V}
      \LDERF \rho^*K_{\RDERF \rho_*\Orb_V}^{\otimes i-1})) \simeq
      \RHom_{\Orb_V}(\LDERF \rho^*P,\rho^*N \otimes^{\LDERF}_{\Orb_V}
      \LDERF \rho^*K_{\RDERF \rho_*\Orb_V}^{\otimes i-1}).
    \]
    By \itemref{proof:cor}, $K_{\RDERF \rho_*\Orb_V}^{\otimes i-1}$
    has $\SHom$-amplitude $[-(e_\rho+1)(i-1),0]$. It follows from
    Lemma \ref{L:hom-amp-properties} that
    $\rho^*N \otimes^{\LDERF}_{\Orb_V} \LDERF \rho^*K_{\RDERF
      \rho_*\Orb_V}^{\otimes i-1}$ has amplitude in
    $[0,(e_\rho+1)(i-1)]$. Further, $\LDERF \rho^*P$ has
    $\SHom$-amplitude $[a,b]$ on $V$. But $V$ is affine, so a short
    inductive argument on amplitude shows that the amplitude of
    $\RHom_{\Orb_V}(\LDERF \rho^*P,\rho^*N \otimes^{\LDERF}_{\Orb_V}
    \LDERF \rho^*K_{\RDERF \rho_*\Orb_V}^{\otimes i-1})$ is contained
    in $[a,b+(e_\rho+1)(i-1)]$. It follows that
    \[
      \amp_{H_P}(N) \subseteq [a,b+(e_\rho+1)(d_\rho-1)].
    \]
    Taking $d_X = (e_\rho+1)(\cd(X)+2)$ gives the claim.
  \end{proof}
  \begin{proof}[Proof of \itemref{proof:loc-glob-2}]
    By Lemma \ref{L:hocolim_pres}, we may write
    $K_{A}^{\otimes i} \simeq \hocolim{s} E_s$, where the
    $E_s \in \DQCOH^{\leq i\cd \pi}(X)$ are pseudo-coherent. Set $m_i = \cd(X)-i+2$. Then, 
  \begin{align*}
    \trunc{>m_i}& \RHom_{\Orb_X}(K_A^{\otimes i},N)\simeq     \trunc{>m_i}\RHom_{\Orb_X}(\hocolim{s} E_s,N)\\
              &\simeq \trunc{>m_i}\holim{s}\RHom_{\Orb_X}(E_s,N)\\
              &\simeq \trunc{>m_i}\holim{s}\trunc{>m_i-1}\RDERF\Gamma(X,\RDERF\SHom_{\Orb_X}(E_s,N)) && (\mbox{Example \ref{ex:milnor_ab}})\\
              &\simeq \trunc{>m_i}\holim{s}\trunc{>m_i-1}\RDERF\Gamma(X,\trunc{>-i+1}\RDERF\SHom_{\Orb_X}(E_s,N)) && \mbox{\cite[Thm.~2.6]{perfect_complexes_stacks}}\\
              &\simeq \trunc{>m_i}\holim{s}\RDERF\Gamma(X,\trunc{>-i+1}\RDERF\SHom_{\Orb_X}(E_s,N))  && (\mbox{Example \ref{ex:milnor_ab}})\\
              &\simeq \trunc{>m_i}\RDERF\Gamma(X,\holim{s}\trunc{>-i+1}\RDERF\SHom_{\Orb_X}(E_s,N)).
  \end{align*}
  If $M\in \DCAT(X)$ and $p\in \Z$, then the $p$th cohomology of $M$
  is the sheafification of the presheaf $V \mapsto \shfcoho^p(V,M)$, as
  $V$ ranges over the affine objects of $X_\lisset$ \cite[Tag
  \spref{0BKV}]{stacks-project}. Now fix an affine object $V$ of
  $X_{\lisset}$; then since
  $\RDERF\SHom_{\Orb_X}(E_s,\Orb_X)\in \DQCOH(X)$ and $V$ is affine we have
  \begin{align*}
    \RDERF \Gamma(V,\holim{s}\trunc{>-i+1}\RDERF\SHom_{\Orb_X}(E_s,N))
    &\simeq \holim{s} \RDERF \Gamma(V,\trunc{>-i+1}\RDERF\SHom_{\Orb_X}(E_s,N))\\
    &\simeq \holim{s} \trunc{>-i+1}\RDERF \Gamma(V,\RDERF\SHom_{\Orb_X}(E_s,N))\\
    &\simeq \holim{s} \trunc{>-i+1}\RHom_{\Orb_V}((E_s)_V,N_V).
  \end{align*}
  But
  \begin{align*}
    \trunc{>-i+1}\holim{s}\RHom_{\Orb_V}((E_s)_V,N_V) &\simeq \trunc{>-i+1}\RHom_{\Orb_V}(\hocolim{s} (E_s)_V,N_V) \\
                                                      &\simeq \trunc{>-i+1}\RHom_{\Orb_V}((K_A^{\otimes i})_V,N_V)\\
                                                      &\simeq 0,                                     
  \end{align*}
  with the vanishing because $K_{A}^{\otimes i}$ has
  $\SHom$-amplitude $[-i(e_\pi+1),-i+1]$. By Example
  \ref{ex:milnor_ab},
  \[
    \holim{s} \trunc{>-i+1}\RHom_{\Orb_V}((E_s)_V,N_V) \simeq 0
  \]
  for all
  affine objects $V$ of $X_{\lisset}$. It follows that
  $\holim{s}\trunc{>-i+1}\RDERF\SHom_{\Orb_X}(E_s,N) \simeq 0$ in
  $\DCAT(X)$. Hence,
  $\trunc{>\cd(X)-i+2}\RHom_{\Orb_X}(K_A^{\otimes i},N) \simeq 0$. Taking $i=\cd(X)+3$, we see that
  \[
    \Hom_{\Orb_X}(K_A^{\otimes \cd(X)+3},\Orb_X) = 0.
  \]
  Hence, $\pi$ is descendable of index $\cd(X)+3$.
\end{proof}
This completes the proof.  \renewcommand\qedsymbol{$\square$}
\end{proof}
\section{Applications}\label{S:applications}
Using Proposition \ref{P:desc-prods} and Theorem \ref{T:desc-non-aff} we can prove the following result on the boundedness of products.
\begin{theorem}\label{T:bounded-products}
  Let $X$ be an algebraic stack of finite cohomological dimension. If
  \begin{enumerate}
  \item $X$ is quasi-compact with affine diagonal or 
  \item $X$ is noetherian and affine-pointed,
  \end{enumerate}
  then $\QCOH(X)$ satisfies \textup{AB}$4^*m$ for some $m$. In particular,
  \[
    \DCAT(\QCOH(X)) \simeq \DQCOH(X).
  \]
\end{theorem}
\begin{proof}
  Let $\pi \colon U \to X$ be a smooth surjection from an affine
  scheme. By \cite[Thm.~C.1]{hallj_neeman_dary_no_compacts},
  $\DCAT^+(\QCOH(X)) \simeq \DQCOH^+(X)$. By Theorem
  \ref{T:desc-non-aff}, $\pi$ is descendable. Also, $\QCOH(U)$ has
  exact products. It now follows from Proposition \ref{P:desc-prods}
  that $\QCOH(X)$ satisfies \textup{AB}$4^*m$ for some $m$. For the
  final equivalence, apply Corollary \ref{C:equiv}.
\end{proof}
Theorems \ref{MT:split-presentation} and \ref{MT:aff-diag} are special
cases of Theorem \ref{T:bounded-products}.
\begin{remark}
  Theorem \ref{T:bounded-products} is mainly interesting in
  characteristic $0$. Indeed, let $p$ be a prime and let $X$ be an
  algebraic stack over $\spec \F_p$ that is quasi-compact with affine
  diagonal or noetherian and affine-pointed. Then the results of
  \cite{algebraization-pairs-stacks}(building on
  \cite{AHR_lunafield,etale_local_stacks}) and
  \cite{hallj_neeman_dary_no_compacts} show that the following
  conditions are equivalent:
  \begin{enumerate}
  \item the reduced connected component of the identity of the geometric
    stabilizer of every point of $X$ is linearly reductive;
  \item $\DQCOH(X)$ is compactly generated;
  \item $\DCAT(\QCOH(X)) \to \DQCOH(X)$ is an equivalence;
  \item $\DCAT(\QCOH(X)) \to \DQCOH(X)$ is full.
  \end{enumerate}
  The implication (3)$\Rightarrow$(4) is trivial. The local structure
  theorems of \cite{algebraization-pairs-stacks} gives that
  (1)$\Rightarrow$(2). We have from
  \cite{hallj_neeman_dary_no_compacts} that (2)$\Rightarrow$(3) and
  (4)$\Rightarrow$(1).
\end{remark}
We now come to the comparison of cohomology result mentioned in the
Introduction. We recall the setup from
\cite[\S2]{hallj_neeman_dary_no_compacts}. Let $f \colon X \to Y$ be a
quasi-compact and quasi-separated morphism of algebraic stacks. Then
the restriction of $(f_{\lisset})_* \colon \MOD(X) \to \MOD(Y)$ to
$\QCOH(X) \subseteq \MOD(X)$ factors through
$\QCOH(Y) \subseteq \MOD(Y)$. This gives us a functor
$(f_{\QCOH})_* \colon \QCOH(X) \to \QCOH(Y)$. Everything is
Grothendieck abelian, so we can derive these functors and obtain a
diagram
\[
  \xymatrix@C+2pc{\DCAT(\QCOH(X)) \ar[r]^{\RDERF (f_{\QCOH})_*} \ar[d]_{\Psi_X} & \DCAT(\QCOH(Y)) \ar[d]^{\Psi_Y} \\ \DCAT(X) \ar[r]_{\RDERF (f_{\lisset})_*} & \DCAT(Y), }
\]
which comes with a natural transformation of functors
\[
  \epsilon_f \colon \Psi_Y \circ \RDERF (f_{\QCOH})_* \Rightarrow \RDERF (f_{\lisset})_*\circ \Psi_X.
\]
If $\Phi_Y \colon \DCAT(Y) \to \DQCOH(Y)$ denotes the right adjoint to
$\Psi_Y$, then we also have a natural transformation of functors
\[
  \epsilon_f^* \colon \RDERF (f_{\QCOH})_* \Rightarrow \Phi_Y \circ
  \RDERF (f_{\lisset})_*\circ \Psi_X.
\]
Now suppose that
both $X$ and $Y$ are quasi-compact with affine diagonal or noetherian
and affine-pointed. Then we have the following results from
\cite[\S2]{hallj_neeman_dary_no_compacts}:
\begin{enumerate}
\item if $M \in \DCAT^+(\QCOH(X))$, then $\epsilon_f(M)$ and
  $\epsilon_f^*(M)$ are equivalences; and
\item if $f$ is concentrated, then $\epsilon_f$ is an isomorphism.
\end{enumerate}
Our application is the following.
\begin{theorem}\label{T:pushforwardagree}
  Let $f \colon X \to Y$ be a morphism of algebraic stacks of finite cohomological dimension. If
  \begin{enumerate}
  \item $X$ and $Y$ are quasi-compact with affine diagonal or 
  \item $X$ and $Y$ are notherian and affine-pointed, 
  \end{enumerate}
  then $\epsilon_f^*$ is an isomorphism.
\end{theorem}
\begin{proof}
  Given what we have established, this is straightforward: $\Psi_Y$ is conservative so it
  suffices to prove that $\Psi_Y(\epsilon_f^*)$ is an isomorphism. We
  may post compose this natural transformation with the adjunction
  $\Psi_Y \circ\Phi_Y \Rightarrow \mathrm{id}$ to obtain:
  \[
    \Psi_Y\circ \RDERF (f_{\QCOH})_* \Rightarrow \Psi_Y \circ \Phi_Y \circ \RDERF (f_{\lisset})_* \circ \Psi_X \Rightarrow \RDERF (f_{\lisset})_* \circ \Psi_X,
  \]
  whose composition is equal to $\epsilon_f$, which is an isomorphism
  by \cite[Cor.~2.2(2)]{hallj_neeman_dary_no_compacts}. Finally, $f$
  is concentrated and so the restriction of $\RDERF (f_{\lisset})_*$
  to $\DQCOH(X)$ factors through $\DQCOH(Y)$. It follows from Theorem
  \ref{T:bounded-products} that
  $\Psi_Y \circ \Phi_Y \circ \RDERF (f_{\lisset})_* \circ \Psi_X
  \Rightarrow \RDERF (f_{\lisset})_* \circ \Psi_X$ is an isomorphism.
\end{proof}
\appendix
\section{Approximations}
In this brief appendix, we collect some technical lemmas that were
employed in \S\ref{S:desc-non-aff}. We begin with noetherian results,
then use the recently established absolute noetherian approximation
for stacks \cite{rydh_absolute_approximation} to deal with the
non-noetherian case.
\begin{lemma}\label{L:countably-generated}
  Let $R$ be a noetherian ring. Let $M$ be a countably generated
  $R$-module. Then every $R$-submodule of $M$ is countably generated.
\end{lemma}
\begin{proof}
  Let $K \subseteq M$ be an $R$-submodule. Write $M=\cup_{n} M_n$ as a
  countable union of finitely generated $R$-modules. Then
  $K_n = M_n \cap K \subseteq M_n$ is finitely generated because $R$
  is noetherian. It follows that $K = \cup_n K_n$ is countably
  generated.
\end{proof}
\begin{lemma}\label{L:countable-mod-noeth}
  Let $X$ be a noetherian algebraic stack. Let $M$ be a quasi-coherent
  $\Orb_X$-module. Assume that $\Gamma(V,M)$ is a countably generated
  $\Gamma(V,\Orb_V)$-module for all affine objects $V$ of $X_{\lisset}$. Then 
  \[
    M[0] \simeq \hocolim{s} E_s[0]
  \]
  in $\DQCOH(X)$, where $E_s \in \COH(X)$ and $E_s \subseteq M$.   
\end{lemma}
\begin{proof}
  If $X$ is affine, then certainly $M\simeq \varinjlim_{s\in \N} E_s$
  in $\QCOH(X)$, where $E_s \in \COH(X)$ and $E_s \subseteq M$. In
  general, the argument in \cite[Prop.~15.4]{MR1771927} extends this
  to noetherian algebraic stacks. Since we have inclusions
  $E_1 \subseteq E_{2} \subseteq \cdots \subseteq M$, and filtered colimits are
  exact,
  $M[0] \simeq (\varinjlim_s E_s)[0] \simeq \hocolim{s} E_s[0]$.
\end{proof}
\begin{lemma}\label{L:approx}
  Let $X$ be a noetherian algebraic stack. Let
  $M \in \DQCOH^{[a,b]}(X)$ be a complex with $\Gamma(V,\COHO{i}(M))$
  countably generated for every $i\in \Z$ and affine object $V$ of
  $X_\lisset$. Then $M\simeq \hocolim{s} E_s$, where
  $E_s \in \DCAT^{[a,b]}_{\COH}(X)$.
\end{lemma}
\begin{proof}
  We prove the result by induction on $n=|b-a|\geq 0$. By shifting, we
  may assume that $b=0$. By Lemma \ref{L:countable-mod-noeth} and the
  inductive hypothesis, we may write
  $\COHO{0}(M) \simeq \hocolim{s}E_s^0$ and
  $\trunc{<0}M \simeq \hocolim{s}F_s$, where $E_s^0 \in \COH(X)$ and
  $F_s \in \DCAT^{[a,-1]}_{\COH}(X)$. We have a distinguished
  triangle:
  \[
    \xymatrix{\trunc{<0}M \ar[r] & M \ar[r] & \COHO{0}(M)[0] \ar[r]^\theta & (\trunc{<0}M)[1]}
  \]
  We will build the $E_s$ by induction on $s$. For each $s$ we obtain
  an induced map
  $\theta_s \colon E_s^0 \to \COHO{0}[0] \xrightarrow{\theta} (\trunc{<0}M)[1] \simeq
  \hocolim{s}F_s[1]$. Since the $F_s \in \DCAT^{[a,-1]}(X)$ and
  $E_s^0 \in \COH(X)$, there exists $r_s$ such that $\theta_s$ factors
  through $\theta'_s \colon E_s^0 \to F_{r_s}[1]$ (this is standard;
  for example, see the proof of \cite[Lem.~1.2]{hallj_coho_bc}). We
  can of course always assume that $s\geq s'$ implies that
  $r_{s} \geq r_{s'}$. Let $E_s$ be the cone of $\theta'_s$. It follows that 
  we have a morphism of distinguished triangles:
  \[
    \xymatrix{F_{r_s} \ar[d] \ar[r] & E_s \ar[r] \ar@{-->}[d]_{\exists e_s}& E_s^0 \ar[d] \ar[r] & F_{r_s} \ar[d] [1]\\
      F_{r_{s+1}} \ar[r] & E_{s+1} \ar[r] & E_{s+1}^0 \ar[r] &
      F_{r_{s+1}}[1].
    }
  \]
  The morphism $e_s$ may even be chosen so that the morphism of
  triangles is \emph{good}, in the sense of
  \cite[Defn.~1.3.14]{MR1812507}. Arguing as in \cite{159896}, we obtain a triangle:
  \[
    \xymatrix{\trunc{<0}M \ar[r] & \hocolim{s} E_s \ar[r] &
      \COHO{0}(M)[0] \ar[r]^{\theta} & (\trunc{<0}M)[1]. }
  \]
  Hence, $\hocolim{s} E_s \simeq M$.
\end{proof}

This last lemma is related to \cite[Tag \spref{0CRQ}]{stacks-project},
where the $E_s$ are taken to be perfect. The following weaker variant,
for stacks, is sufficient for us.
\begin{lemma}\label{L:hocolim_pres}
  Let $\pi \colon U \to X$ be a concentrated, flat, and finitely
  presented morphism of algebraic stacks. Form the distinguished triangle:
  \[
    \xymatrix{K \ar[r] & \Orb_X \ar[r] & \RDERF \pi_*\Orb_U \ar[r] & K[1].}
  \]
  Then $K^{\otimes i} \simeq \hocolim{s} E_s$, where the
  $E_s \in \DQCOH^{\leq i\cd(\pi)}(X)$ are pseudo-coherent, in the following
  situations:
  \begin{enumerate}
  \item $\pi$ is tame (e.g., representable); or
  \item $\pi$ has affine stabilizers and $X$ has equicharacteristic;
    or
  \item \label{LI:hocolim_pres:countable} $X$ is noetherian.
  \end{enumerate}
\end{lemma}
\begin{proof}
  By absolute noetherian approximation
  \cite{rydh_absolute_approximation} and tor-independent base change
  \cite[Cor.~4.13]{perfect_complexes_stacks}, it suffices to establish the
  result in case \itemref{LI:hocolim_pres:countable}. By Lemma
  \ref{L:approx}, it remains to prove that
  $\shfcoho^0(V,\RDERF^p \pi_*\Orb_U)$ is countably generated for
  every affine object $V$ of $X_{\lisset}$. This is local on $X$, so
  we may assume that $X=\spec A$ is noetherian and affine. Let
  $U_\bullet \to U$ be a smooth hypercovering, where the
  $U_i = \spec B_i$ are affine. Then $A \to B_i$ is of finite type, so
  $B_i$ is a countably generated $A$-module. Hypercohomology says that
  $\shfcoho^p(U,\Orb_U)$ is computed as the cohomology of the complex
  with $B_i$ sitting in degree $i$. It follows from Lemma
  \ref{L:countably-generated} that the cohomology of this complex is
  countably generated as an $A$-module.
\end{proof}

\bibliography{bibtex_db/references} \bibliographystyle{bibtex_db/dary}
\end{document}